\documentclass{amsart}
\usepackage{amsfonts,amsmath,amssymb,color}
\usepackage{amsrefs}
\usepackage{hyperref}

\newtheorem{theo}{Theorem}[section]
\newtheorem{prop}[theo]{Proposition}
\newtheorem{defi}{Definition}[section]
\newtheorem{lemme}[theo]{Lemma}
\newtheorem{corol}[theo]{Corollary}
\theoremstyle{remark}
\newtheorem{remark}[theo]{Remark}
\newcommand{\be}{\begin{equation*}}
\newcommand{\ee}{\end{equation*}}
\newcommand{\ben}{\begin{equation}}
\newcommand{\een}{\end{equation}}
\newcommand{\bal}{\begin{aligned}}
\newcommand{\eal}{\end{aligned}}

\newcommand{\pui}{\frac{n-2}{2}}
\newcommand{\ue }{u_\ve }
\newcommand{\xe}{x_\ve}
\newcommand{\he}{h_\ve}

\newcommand{\me}{\mu_\ve}
\newcommand{\R}{\mathbb{R}}
\newcommand{\rn}{\mathbb{R}^n}
\newcommand{\rr}{\mathbb{R}}
\newcommand{\sph}{\mathbb{S}}
\newcommand{\ve}{\varepsilon}
\newcommand{\vp}{\varphi}
\newcommand{\pe}{p_{\varepsilon}}

\def \crit{2^\star}

\numberwithin{equation}{section}

\title[]{One-bubble nodal blow-up for asymptotically critical stationary Schr\"odinger-type equations}

\author{Bruno Premoselli}

\address{Bruno Premoselli, Universit\'e Libre de Bruxelles, Service d'analyse, CP 218, Boulevard du Triomphe, B-1050 Bruxelles, Belgique.}
\email{bruno.premoselli@ulb.be}

\author{Fr\'ed\'eric Robert}

\address{Fr\'ed\'eric Robert, Universit\'e de Lorraine, CNRS, IECL, F-54000 Nancy, France}
\email{frederic.robert@univ-lorraine.fr}

\thanks{The first author was supported by the FNRS CdR grant J.0135.19, the Fonds Th\'elam and an ARC Avanc\'e 2020 grant.} 

\date{April 24th, 2024}

\begin{document}

\begin{abstract}
We investigate in this work families $(u_\ve)_{\ve >0}$ of sign-changing blowing-up solutions of asymptotically critical stationary nonlinear Schr\"odinger  equations of the following type:
$$\Delta_g u_\ve + h_\ve u_\ve = |u_{\ve}|^{p_\ve-2} u_\ve $$
in a closed manifold $(M,g)$, where $h_\ve$ converges in $C^1(M)$. Assuming that $(u_\ve)_{\ve >0}$ blows-up as \emph{a single sign-changing bubble}, we obtain necessary conditions for blow-up that constrain the localisation of blow-up points and exhibit a strong interaction between $h$, the geometry of $(M,g)$ and the bubble itself. These conditions are new and are a consequence of the sign-changing nature of $u_\ve$. \end{abstract}

\maketitle

\section{Introduction}

\subsection{Statement of the main results}

Let $(M^n,g), n \ge 3$ be a smooth, connected and closed manifold, where closed means compact without boundary. We study in this paper sign-changing solutions $u\in C^2(M)$ of the   equation
\ben \label{Intro:eq0}
\Delta_g u  + h u  = |u |^{p -2} u\hbox{ in }M
\een
where $\Delta_g = - \text{div}_g(\nabla \cdot)$ is the Laplace-Beltrami operator,  $h\in C^1(M)$ and $2<p\leq \crit$ with $\crit  = \frac{2n}{n-2}$. When $h\equiv \frac{n-2}{4(n-1)} S_g$, where $S_g$ is the scalar curvature of $(M,g)$, $p= \crit$ and $u>0$, \eqref{Intro:eq0} is the celebrated Yamabe equation. We let $H^1(M)$ be the completion of $C^\infty(M)$ for $u\mapsto \Vert u\Vert_{H^1}:=\Vert u\Vert_2+\Vert\nabla u\Vert_2$. It has been known since the seminal work of Struwe \cite{Struwe}  that families of solutions to \eqref{Intro:eq0} that are uniformly bounded in $H^1(M)$ may develop concentration phenomena in the form of ``bubbles" (see also Druet-Hebey-Robert \cite{DruetHebeyRobert} for the case of a compact Riemannian manifold). These ``bubbles'' correspond to a loss of compactness of the solutions, and understanding the conditions in which they may appear and their mutual interactions has been the subject of many works in the last decades. 

\smallskip\noindent Despite the abundance of contributions for positive solutions of \eqref{Intro:eq0}, the literature for sign-changing solutions is less developed. In the present work we take a step in this direction and study sign-changing solutions of \eqref{Intro:eq0} that blow-up as \emph{a single sign-changing bubble}. In the sequel, we fix $r_g\in (0, i_g(M))$ where $i_g(M)>0$ is the injectivity radius of $(M,g)$. Throughout this paper $\chi \in C^\infty_c(\R^n)$ will denote a cutoff function such that $\chi(t)=1$ if $|t|\leq r_g/2$ and $\chi(t)=0$ if $|t|\geq r_g$. 
We say that a family $(B_\ve)_{\ve >0}$ of functions in $H^1(M)$ is a family of bubbles if there exists a non-zero solution $V$ to
\ben \label{yamabe:intro}
\Delta_{\xi} V = |V|^{\crit -2} V  \text{ in } \R^n , \quad  V\in  D^{1,2}(\rn),
\een
where $D^{1,2}(\rn)$ is the completion of $C^\infty_c(\rn)$ for the norm $u \mapsto \Vert\nabla u \Vert_2$,
and families of points $(x_\ve)_{\ve>0}$ in $M$ and of positive real numbers $(\mu_\ve)_{\ve >0}$ such that 
$$B_\ve(x) -  \chi(d_g(\cdot,\xe ))\me^{-\frac{n-2}{2}}V\left(\frac{\exp_{x_\ve}^{-1}(\cdot)}{\me}\right)  \to 0 \text{ in } H^1(M) $$
as $\ve \to 0$. Here $\exp_{x_\ve}$ is the exponential chart at $x_\ve$ and $\xi$ is the Euclidean metric, so that $\Delta_\xi = - \sum_{i=1}^n \partial_i^2$. We will say that such a family $(B_\ve)_{\ve >0}$ is centered at $x_\ve$ and of radius $\mu_\ve$ and modeled on $V$, and $V$ itself will sometimes be referred to as the bubble. A more precise definition is given in Definition \ref{def:bubble} and properties of such families are investigated in Section \ref{sec:weyl} below. All the solutions considered in this paper are allowed to change sign unless mentioned otherwise. 

\medskip\noindent  Our main results establish a necessary condition for the existence of blowing-up solutions of \eqref{Intro:eq0} that blow-up with a single sign-changing bubble. We introduce one last notation as follows: if $B = (B_\ve)_{\ve >0}$ is a family of bubbles centered at $x_\ve$, of radius $\mu_\ve$ and modeled on $V$,  
we define 
$$ \begin{aligned}
\hbox{Weyl}_g\otimes B&:= \frac{4n}{3(n-2)^2}  \left( \int_{\R^n} |V|^{\crit }dx \right)^{-1}\times \\
 &\hbox{Weyl}_g(x_0)_{i\alpha j\beta}    \int_{\rn}x^\alpha x^\beta \partial_{ij}^2 V \left(\frac{n-2}{2}V+x^l\partial_l V\right)\, dx
\end{aligned} $$
where $x_0 = \lim_{\ve \to 0}  x_\ve$. We first state a result when $n \ge 5$.

 \begin{theo} \label{maintheo}
Let $(M^n,g), n \ge 5$ be a smooth, connected and closed manifold, let $h, (h_\ve)_{0 < \ve \le 1}\in C^1(M)$ such that $\lim_{\ve\to 0}h_\ve=h$ in $C^1(M)$ and $\Delta_g + h$ is coercive. Let $(p_\ve)_{0< \ve \le 1}\in (2, \crit]$ be such that $\lim_{\ve \to 0} p_\ve = \crit $. Let  $(u_\ve)_{0< \ve \le 1}$ be a family of solutions of 
\ben \label{Intro:eq1}
\Delta_g u_\ve + h_\ve u_\ve = |u_{\ve}|^{p_\ve-2} u_\ve \quad \text{ in } M.
\een
We assume that $(u_\ve)_{0< \ve \le 1}$ satisfies 
\begin{equation}\label{one:bubble}
u_\ve =cB_\ve+o(1)\hbox{ in }H^1(M)
\end{equation}
for some $c>0$, where $B=(B_\ve)_{\ve>0}$ is a bubble centered at $(\xe)_\ve\in M$ and with radius $(\me)_\ve\in (0,+\infty)$ and modeled on $V\in D^{1,2}(\rn)$ as in Definition \ref{def:bubble} below. 

\medskip

\noindent Then $c=1$, $\lim_{\ve\to 0}\me^{\crit-p_\ve}=1$, the following limit exists:
$$ \Lambda := \lim_{\ve\to0} \frac{2^*-p_\ve}{\mu_\ve^2} \ge 0,$$
 and $h$ satisfies: 
\begin{equation}\label{condnecessaire:5}
h(x_0) - \frac{n-2}{4(n-1)} S_g(x_0) = \frac{(n-2)^2}{4n}  \frac{   \int_{\R^n} |V|^{\crit }dx}{\int_{\R^n} V^2dx } \Big(\Lambda - \hbox{Weyl}_g\otimes B \Big),
\end{equation}
where we have let $x_0:=\lim_{\ve\to 0}\xe$. 
 \end{theo}
The new term $\hbox{Weyl}_g \otimes B$ appearing in \eqref{condnecessaire:5} is well-defined when $n \ge 5$ and we prove in Section \ref{sec:weyl} that it does not depend on the representative chosen for the family $(B_\ve)_{\ve >0}$ appearing in \eqref{one:bubble} (see Proposition \ref{prop:unique:bubble} below). 

\medskip\noindent In dimensions $3$ and $4$ we prove the following analogue of Theorem \ref{maintheo}: 

 \begin{theo} \label{maintheo2}
Let $(M^n,g), n =3,4$ be a smooth, connected and closed manifold, let $h, (h_\ve)_{0 < \ve \le 1},(p_\ve)_{0< \ve \le 1}\in (2, \crit]$ be as in the statement of Theorem \ref{maintheo}. Let  $(u_\ve)_{0< \ve \le 1}$ be a family of solutions of \eqref{Intro:eq1} satisfying again \eqref{one:bubble} for some $c>0$, where $B=(B_\ve)_{\ve>0}$ is a bubble centered at $(\xe)_\ve\in M$ and with radius $(\me)_\ve\in (0,+\infty)$ and modeled on $V\in D^{1,2}(\rn)$ as in Definition \ref{def:bubble} below. 

\medskip \noindent Then $c=1$, $\lim_{\ve\to 0}\me^{\crit-p_\ve}=1$, the following limits exist:
$$\Lambda : = \lim_{\ve \to 0} \left \{\begin{aligned} & \frac{6-p_\ve}{\mu_\ve} & n=3 \\ & \frac{4 - p_\ve}{\me^2 \ln \frac{1}{\me}}  & n=4 \end{aligned} \right \} \ge 0 , $$
and $h$ satisfies
\begin{equation}\label{condnecessaire:4}
 \frac{\left(\int_{\R^4}|V|^{2}V\, dx\right)^2}{\omega_3  \int_{\R^4}|V|^{4}\, dx} \left(  h(x_0) - \frac{1}{6} S_g(x_0) \right) = \Lambda \hbox{ when }n=4
\end{equation}
\begin{equation}\label{condnecessaire:3}
 -6 \frac{\left( \int_{\R^3}|V|^{4}V\, dx\right)^2}{ \int_{\R^3}|V|^{6}\, dx} m_{h}(x_0) = \Lambda \hbox{ when }n=3,
\end{equation}
where as before we have let $x_0:=\lim_{\ve\to 0}\xe$. In \eqref{condnecessaire:4} $\omega_3$ is the area of the $3$-sphere and in  \eqref{condnecessaire:3} $m_h(x_0)\in\rr$ is the mass of the Green's function of $\Delta_g +h$ at $x_0\in M$ (see Definition \ref{def:mass} below). 
 \end{theo}
A few comments on Theorems \ref{maintheo} and \ref{maintheo2} are in order. Assumption \eqref{one:bubble} implies that $(u_\ve)_{\ve >0}$ blows-up as a single, possibly sign-changing bubble. In other words, \eqref{one:bubble} is a Struwe decomposition with a single bubbling profile. No assumption is made on the energy of $V$ (except that it is finite) and $V$ is not assumed to be non-degenerate in the sense of Duyckaerts-Kenig-Merle \cite{DuyckaertsKenigMerle}. Assumption \eqref{one:bubble} should be regarded, insofar as sign-changing solutions are considered, as the analogue of the classical notion of isolated and simple blow-up point for positive solutions introduced in Schoen \cite{SchoenPreprint}. 

\medskip\noindent Theorems \ref{maintheo} and \ref{maintheo2} are not new for \emph{positive} solutions. It has been known since Obata \cite{Obata} and Struwe \cite{Struwe} that when $(u_\ve)_{\ve>0}$ is a family of  \emph{positive} solutions of \eqref{Intro:eq1} the bubble $V$ is positive and radial and given by \eqref{B0} below. In this case, and when $n \ge 5$, the additional term $\hbox{Weyl}_g \otimes B$ in \eqref{condnecessaire:5} vanishes by the symmetries of $\hbox{Weyl}_g$ (see Proposition \ref{prop:WB} below), and condition \eqref{condnecessaire:5} becomes 
 $$ h(x_0) - \frac{n-2}{4(n-1)} S_g(x_0)= C(n)\Lambda \ge 0, \quad \hbox{ with }\Lambda:=\lim_{\ve\to 0}\frac{\crit -p_\ve}{\mu_\ve^2}$$
for some dimensional constant $C(n)>0$. When $n=3,4$, \eqref{condnecessaire:4} and \eqref{condnecessaire:3} show that 
\begin{equation*}
\begin{aligned}
 h(x_0) - \frac{1}{6} S_g(x_0)  = C(4) \Lambda \hbox{ when }n=4 \quad \text{ and } \quad 
 m_{h}(x_0)  = - C(3)\Lambda \hbox{ when }n=3,
\end{aligned} 
\end{equation*}
where $\Lambda$ is as in the statement of the Theorem. When $p_\ve = \crit $ for all $\ve$, and for positive finite-energy solutions, these conditions have been known since the work of Li-Zhu \cite{LiZhu} and Druet \cites{DruetJDG, DruetYlowdim}. 

\smallskip\noindent If $(u_\ve)_{\ve>0}$ is a blowing-up family of \emph{sign-changing} solutions of \eqref{Intro:eq1} with a single bubbling profile as in \eqref{one:bubble}, however, the bubble $V$ is sign-changing in general and  \eqref{condnecessaire:5}, \eqref{condnecessaire:4} and \eqref{condnecessaire:3} display new phenomena. When $n \ge 5$ the additional term $\hbox{Weyl}_g \otimes B$ cannot be expected to vanish \emph{a priori}. We construct indeed in Section \ref{sec:weyl} examples of manifolds $(M,g)$ and of solutions $V$ of \eqref{yamabe:intro} such that $\hbox{Weyl}_g \otimes B \neq 0$ (see Proposition \ref{prop:weyl:negative} below). The term $\hbox{Weyl}_g \otimes B$ is thus a striking new feature of sign-changing blow-up that does not appear for positive solutions. When $n \ge 5$, condition \eqref{condnecessaire:5} is entirely new and highlights a strong interaction between $h$, the geometry of $(M,g)$ and the limiting bubble $V$ itself. When $n=3,4$, \eqref{condnecessaire:4} and  \eqref{condnecessaire:3} involve $\int_{\R^n} |V|^{2^*-2}V dx$. But if $V$ is a sign-changing solution of \eqref{yamabe:intro} for $n \ge 3$, the latter integral may vanish if $V$ decays strongly at infinity (see \eqref{deflambda} below). When this is the case conditions \eqref{condnecessaire:4} and \eqref{condnecessaire:3} indicate that, in dimensions $3$ and $4$, the interaction between $B_\ve$ and the geometry of $(M,g)$ takes place at a higher-order than it did for positive solutions.

\subsection{Consequences of Theorem \ref{maintheo} on the blow-up picture of \eqref{Intro:eq1}}

\noindent Theorems \ref{maintheo} and \ref{maintheo2} shed new light on the \emph{stability} of the set of sign-changing solutions of \eqref{Intro:eq1}. We say that equation \eqref{Intro:eq0}, with $p = 2^*$, is stable if for every family $(h_{\ve})_{\ve >0}$ converging to $h$, every energy-bounded family of solutions $(u_\ve)_{\ve >0}$ of \eqref{Intro:eq1} converges in $C^2(M)$, up to a subsequence. Druet \cite{DruetJDG} proved that, for \emph{positive} solutions, stability holds provided $h \neq \frac{n-2}{4(n-1)} S_g$ everywhere in $M$ (with a caveat in dimensions $3$ and  $6$, see \cite{DruetJDG}). This result was generalised to the case of sign-changing solutions, when $n \ge 7$ and when $(M,g)$ is locally conformally flat, in Premoselli-V\'etois \cite{PremoselliVetois2}. When $(M,g)$ is not locally conformally flat, and at least when $n \ge 5$, however, \eqref{condnecessaire:5} shows that necessary conditions for a one bubble blow-up involve $h$, the geometry of $(M,g)$ and the possible limiting bubble $V$ itself. The set of sign-changing solutions of \eqref{yamabe:intro} is still poorly understood and no classification result is known, but examples of large-energy sequences of solutions have been constructed by Del Pino-Musso-Pacard-Pistoia \cite{DelPinoMussoPacardPistoia1}, Ding \cite{ding},  Medina-Musso \cite{MedinaMusso}, Medina-Musso-Wei \cite{MedinaMussoWei}. The wealth of solutions of \eqref{yamabe:intro} is the main obstacle to understanding the term $\hbox{Weyl}_g \otimes B$ in \eqref{condnecessaire:5}. It seems therefore highly unlikely, even in the single-bubble case, that stability for sign-changing solutions of \eqref{Intro:eq1} can be enforced solely by a global assumption on $h$, as was the case for positive solutions. 

\medskip\noindent Theorems \ref{maintheo} and \ref{maintheo2}, however, allow us to rule out limiting bubbles $V$ that may appear in the blow-up in some cases. We start with an example in dimensions $3$ and $4$. If $V$ is a solution of \eqref{yamabe:intro} it satisfies 
$$ V(x) = \frac{\lambda(V)}{|x|^{n-2}} +O \left(  |x|^{1-n} \right)   \quad \text{ as } |x| \to + \infty, $$
where $(n-2) \omega_{n-1} \lambda(V) = \int_{\R^n} |V|^{\crit -2}V dx$ (see \eqref{eq:expansion} and \eqref{deflambda} below). Bubbles decaying at infinity as $|x|^{2-n}$ satisfy $\lambda(V)>0$: this is the case for the positive bubbles given by \eqref{B0} and for the bubbles of \cite{DelPinoMussoPacardPistoia1}. An immediate corollary of Theorem \ref{maintheo2} is as follows:

\begin{corol} \label{corol:mass:3}
Let $(M^n,g)$, $n=3,4$, be a smooth, closed, connected Riemannian manifold and $h, (h_\ve)_{0 < \ve \le 1}\in C^1(M)$ be such that $\lim_{\ve\to 0}h_\ve=h$ in $C^1(M)$ and $\Delta_g + h$ is coercive. Let $(p_\ve)_{0< \ve \le 1}\in (2, \crit]$ be such that $\lim_{\ve \to 0} p_\ve = \crit $. Let  $(u_\ve)_{0< \ve \le 1}$ be a family of solutions of \eqref{Intro:eq1} satisfying \eqref{one:bubble}, where $B_\ve$ is a family of bubbles centered at $x_\ve$, of radius $\me$ and modeled on $V$. Assume that
\begin{itemize}
\item the mass of $\Delta_g +h$ is positive at every point of $M$ if $n=3$
\item $h < \frac{1}{6} S_g$ everywhere in $M$ if $n=4$.
\end{itemize} 
Then $\int_{\R^n} |V|^{2^*-2}V dx = 0$. In particular, $V(x) = O(|x|^{1-n})$ as $|x| \to + \infty$ and $V$ is neither positive nor one of the bubbles of \cite{DelPinoMussoPacardPistoia1}.
\end{corol}

In dimensions $n \ge 5$ \eqref{condnecessaire:5} does not provide any information on the decay of $V$ at infinity, but we can still rule out the existence of symmetric bubbles in some cases. We recall that Ding \cite{ding} constructed a family of solutions of \eqref{yamabe:intro}, possessing $O(p) \times O(n-p)$ symmetry for $2 \le p \le n-1$. We show that Ding solutions cannot appear as bubbles in a single-bubble blow-up on symmetric manifolds:

\begin{corol} \label{corol:ding:5}
Let $(M^n,g) = (\mathbb{S}^p \times \mathbb{S}^q, g_p \times g_q)$, where $g_p$ and $g_q$ are the round metrics on $\mathbb{S}^p$ and $\mathbb{S}^q$, $p,q \ge2$ and $n=p+q \ge 5$. Let $h, (h_\ve)_{0 < \ve \le 1}\in C^1(M)$ be such that $\lim_{\ve\to 0}h_\ve=h$ in $C^1(M)$ and $\Delta_g + h$ is coercive. Let $(p_\ve)_{0< \ve \le 1}\in (2, \crit]$ be such that $\lim_{\ve \to 0} p_\ve = \crit $. Let  $(u_\ve)_{0< \ve \le 1}$ be a family of solutions of \eqref{Intro:eq1} satisfying \eqref{one:bubble}, where $B_\ve$ is a family of bubbles centered at $x_\ve$, of radius $\me$, modeled on $V$. Assume that $h \le \frac{n-2}{4(n-1)}S_g$ everywhere. Then
$$ \hbox{Weyl}_g \otimes B \ge 0. $$
As a consequence, $V$ cannot be a Ding solution. 
\end{corol}

\begin{proof}
Since $h \le \frac{n-2}{4(n-1)}S_g$ everywhere, \eqref{condnecessaire:5} shows that $ \hbox{Weyl}_g \otimes B \ge 0$. But as shown in Corollary \ref{corol:weyl:negatif} below, $\hbox{Weyl}_g \otimes B <0$ if $V$ is a Ding solution. 
\end{proof}
Corollaries \ref{corol:mass:3} (when $n=3$) and \ref{corol:ding:5} both apply to one-bubble sign-changing blow-up for the sign-changing Yamabe equation:
$$ \Delta_g u + \frac{n-2}{4(n-1)}S_g u = |u|^{2^*-2}u. $$
Solutions of this equation naturally appear as minimisers for the second conformal eigenvalue (see Ammann-Humbert \cite{AmmannHumbert}), and their stability properties shed new light on the set of possible minimisers (see Premoselli-V\'etois \cite{PremoselliVetois3}). Other stability results for sign-changing solutions are in Premoselli-V\'etois \cites{PremoselliVetois, PremoselliVetois2}. Examples of sign-changing families of solutions of \eqref{Intro:eq1} have been constructed in Bonheure-Casteras-Premoselli \cite{BonheureCasterasPremoselli}, Micheletti-Pistoia-V\'etois \cite{MichelettiPistoiaVetois}, Pistoia-V\'etois \cite{PistoiaVetois} and Robert-V\'etois \cites{RobertVetois3, RobertVetois4}. Deng-Musso-Wei \cite{DengMussoWei} have considered an innovative  construction of a family of sign-changing solutions $(u_\ve)_{\ve >0}$ to the \eqref{Intro:eq1} with $h_\ve\equiv h$ and $p_\ve:=\crit-\ve$ that blows-up like a single bubble as in \eqref{one:bubble}. We discuss similar constructions in Section \ref{sec:final}.

\medskip\noindent The structure of the paper is as follows. In Section \ref{Sec2} we introduce a few properties of sign-changing solutions of \eqref{yamabe:intro} and highlight the main differences with respect to the positive case. Section \ref{sec:weyl} contains the formal definition and properties of the families of bubbles that we investigate here. We also investigate the term $\hbox{Weyl}_g \otimes B$ there and construct examples of manifolds and bubbles where $\hbox{Weyl}_g \otimes B \neq 0$. Section \ref{Sec3} is the core of the analysis of this paper and we prove there that relation \eqref{one:bubble} can be improved into optimal pointwise bounds on $u_\ve$. This is the content of Proposition \ref{prop:C0} and is the main ingredient in the proof of Theorems \ref{maintheo} and \ref{maintheo2}. Finally, Theorems \ref{maintheo} and \ref{maintheo2} are proven in Section \ref{Sec5}. 

\medskip

\textbf{Acknowledgements:} the authors would like to thank P.-D. Thizy and J. V\'etois for many fruitful discussions in the early version of this work.

 \section{Finite-energy nodal solutions of the Yamabe equation in $\R^n$} \label{Sec2}
We let $\Sigma$ be the set of non-zero finite-energy solutions of \eqref{yamabe:intro}:
\ben \label{defSigma}
\Sigma = \big \{ V \in D^{1,2}(\R^n) \backslash \{ 0 \}, \quad \Delta_{\xi} V =  |V|^{\crit -2} V \big \}. 
\een
Regularity theory for critical equations (see eg Trudinger \cite{Trudinger}) together with standard elliptic theory, shows that there exists $\alpha_n>0$ such that $V\in C^{3,\alpha_n}_{loc}(\R^n)$ for all $V\in \Sigma$.  If $V \in \Sigma$ is positive it is equal by Obata \cite{Obata}, up to translations and rescalings, to
\begin{equation}\label{B0}
B_0^+(x)=\left(1+ \frac{|x|^2}{n(n-2)}\right)^{-\frac{n-2}{2}}, \quad x \in \R^n.
\end{equation}
By Caffarelli-Gidas-Spruck \cite{CaGiSp} this result remains true for all positive solutions of \eqref{yamabe:intro}, without the assumption that they belong to $D^{1,2}(\R^n)$. Note also that, up to dilation, $B_0^+$ is the only non-zero \emph{radial} solution of \eqref{yamabe:intro} which belongs to $D^{1,2}(\R^n)$ by a simple application of Pohozaev's identity. 

\medskip\noindent 
If $V \in D^{1,2}(\R^n)$ we define its Kelvin transform as
\ben \label{eq:Kelvin}
V^*(x)= \frac{1}{|x|^{n-2}} V \left(  \frac{x}{|x|^2}\right)   \quad \text{ for a.e. } x \in \R^n\backslash \{0\}.
\een
It is well-known (see for instance Duyckaerts-Kenig-Merle \cite[Proposition 3.1]{DuyckaertsKenigMerle}) that $V \mapsto V^*$ defines an isometry of $D^{1,2}(\R^n)$ and of $L^{\crit }(\R^n)$ and that 
$$\Delta_{\xi} V^*(x) = \frac{1}{|x|^{n+2}} \Delta_{\xi} V\left( \frac{x}{|x|^2} \right)$$
in $\R^n \backslash \{0\}$, provided $V$ is of class $C^2$. As a consequence we have $V^* \in \Sigma$ whenever $V \in \Sigma$ (see again \cite{DuyckaertsKenigMerle}) and the regularity theory for \eqref{yamabe:intro} shows that $V^* \in C^{3}(\R^n)$ when $V \in \Sigma$. The following result provides a precise description of the behavior at infinity of the elements in $\Sigma$:

\begin{lemme} \label{lemme:expansion}
Let $V \in \Sigma$. There exist $\lambda(V)\in \R$ and $\alpha(V) \in \R^n$ such that the following asymptotic expansion holds as $|x|\to + \infty$ and can be differentiated:
\ben \label{eq:expansion}
 \begin{aligned}
 V(x) & = \frac{\lambda(V)}{|x|^{n-2}} + \frac{\langle \alpha(V), x\rangle}{|x|^n} + O \left(  \frac{1}{|x|^n}\right)   .
  \end{aligned} 
 \een
 \end{lemme}
When we say that expansion \eqref{eq:expansion} can be differentiated we mean that the following holds as $|x| \to + \infty$: for any $1 \le i \le n$,
\ben \label{eq:expansionder} \begin{aligned}
\partial_i V(x) & = - (n-2)\lambda(V)\frac{x_i}{|x|^{n}} + \frac{\alpha(V)_i |x|^2 - n \langle \alpha(V), x\rangle x_i}{|x|^{n+2}} \\
 &+ O \left(  \frac{1}{|x|^{n+1}}\right)   .\\
 \end{aligned}
 \een

\begin{proof}
Let $V^*$ be, as before, the Kelvin transform of $V$. It is of class $C^3$ in $\R^n$ and we can consider its Taylor expansion at $0$: there exist $\lambda(V) \in \R$ and $\alpha(V) \in \R^n$ such that, as $y \to 0 $, and for $1 \le i \le n$,
\ben \label{eq:expansion1}
\begin{aligned} 
V^*(y) & = \lambda(V) + \langle \alpha(V), y\rangle + O(|y|^2) \\
\partial_i V^*(y) & = \alpha(V)_i + O(|y|)
\end{aligned} 
\een
hold. Since for any $x \neq 0$ we have $ V(x) = \frac{1}{|x|^{n-2}} V^* \left(  \frac{x}{|x|^2} \right)  $, \eqref{eq:expansion} follows from \eqref{eq:expansion1}. To prove \eqref{eq:expansionder} we write that by the chain rule, for $x \neq 0$,
\begin{equation} \label{chainrulekelvin} \partial_i V(x) = - (n-2) \frac{x_i}{|x|^{n}} V^* \left(  \frac{x}{|x|^2} \right)    + \sum_{j=1}^n \frac{1}{|x|^{n-2}} \left(  \frac{\delta_{ij}}{|x|^2} -2 \frac{x_ix_j}{|x|^4} \right)   \partial_j V^* \left(  \frac{x}{|x|^2} \right) 
\end{equation}
holds, and we conclude using again \eqref{eq:expansion1}. 
\end{proof}
As a consequence of \eqref{eq:expansion} and \eqref{eq:expansionder}, there exists $C = C(n,V)$ such that 
\ben \label{eq:decayV}
|V(x)| + (1+|x|)|\nabla V(x)| + (1+|x|)^2|\nabla^2 V(x)| \le C (1 + |x|)^{2-n}
\een
for all $x \in \R^n$. This can also be shown by the conformal invariance of \eqref{yamabe:intro}, see e.g. Premoselli \cite[Lemma 2.2]{Premoselli13}. As shown in Premoselli \cite[Lemma 2.1]{Premoselli13},  we have
\ben \label{deflambda}
\lambda(V) = \frac{1}{(n-2) \omega_{n-1} }\int_{\R^n} |V|^{\crit -2}V dx \quad \hbox{ for all } V \in \Sigma.
\een
If $B_0^+$ is as in \eqref{B0} we have $\lambda(B_0^+) = (n(n-2))^{\frac{n-2}{2}}$. The solutions constructed in \cite{DelPinoMussoPacardPistoia1} also satisfy $\lambda(V) \neq 0$. As proven in Lemma \ref{lemme:expansion}, $\lambda(V)= V^*(0)$. As a consequence, examples where $\lambda(V) = 0$ are easily obtained by Kelvin-transforming sign-changing solutions of $\Sigma$ at a point where they vanish. 

\begin{remark} \label{rem:decay}
An accurate inspection of the proof of Lemma \ref{lemme:expansion} shows that we can state a slightly more precise result. Let $V \in \Sigma$.
\begin{itemize}
\item Assume that $n=3,4$. Then there exists $k \in \mathbb{N}$ and a non-zero homogeneous polynomial $P$ of degree $k$ such that 
\be 
 V(x) = \frac{P(x)}{|x|^{n-2+2k}} + O \left(  \frac{1}{|x|^{n-1+k}}\right)   \quad \text{ as } |x| \to + \infty .
 \ee

\item Assume that $n =5$. Then there exists $k \in \{0, \dots, 4\}$ and a homogeneous polynomial $P$ of degree $k$ such that 
\be 
 V(x) = \frac{P(x)}{|x|^{3+2k}} + \left \{ \begin{aligned} & O \left( \frac{1}{|x|^{4+k}}\right)  & \text{ if } k \le 3 \\
 & O \left( \frac{1}{|x|^{7 + \frac13}}\right) & \text{ if } k = 4  
 \end{aligned} \right.  \quad \text{ as } |x| \to + \infty .
 \ee

\item Assume finally that $n \ge 6$. Let $0 <\alpha < 1$ if $n = 6$ and $0 < \alpha \le \frac{4}{n-2}$ if $n \ge 7$. Then there exists $k \in \{0, \dots, 3\}$ and a homogeneous polynomial $P$ of degree $k$ such that $V$ satisfies 
\be 
 V(x) = \frac{P(x)}{|x|^{n-2+2k}} + O \left(  \frac{1}{|x|^{n-2+k+\alpha}}\right)   \quad \text{ as } |x| \to + \infty.
 \ee

 \end{itemize}
All these expansions can be differentiated. In dimensions $n=3,4$, $P$ is the first non-zero homogeneous polynomial in the Taylor expansion of $V^*$ at $0$, which always exists by classical finite continuation results (see e.g. Aronszajn \cite{Aronszajn}) since $V \in D^{1,2}(\R^n) \backslash \{0\}$. In dimensions $n \ge 5$, $P$ may vanish. Examples of solutions in any dimension $n\ge 3$ for which the degree of $P$ is larger or equal than $2$ are still unknown and the question of their existence was raised in Duyckaerts-Kenig-Merle\cite{DuyckaertsKenigMerle}.
\end{remark}

\section{Riemannian preliminaries}\label{sec:weyl}

\subsection{Sign-changing families of bubbles in $M$}
We define the notion of Riemannian bubble that we will consider in this work:

\begin{defi}\label{def:bubble} Let $(\xe )_{\ve>0}\in M$ and $(\me)_{\ve>0}\in \rr_{>0}$ be families of points in $M$ and positive real numbers, with $\mu_{\ve} \to  0$ as $\ve \to 0$. We let $x_0:=\lim_{\ve\to 0}\xe $. Let $V\in \Sigma$ be fixed and let $B = (B_\ve)_{\ve >0}$ be a family of functions in $H^1(M)$. We say that $B$ is a \emph{family of bubbles centered at $x_\ve$ with radius $\mu_\ve$ and modeled on $V$ } if there exists a family of charts $(\varphi_\ve)_{\ve\geq 0}$ such that $\varphi_\ve: B_g(x_0, r_g) \to \R^n$ is normal at $\xe$ for all $\ve>0$, $\lim_{\ve\to 0}\varphi_\ve=\varphi_0$ in $C^k$ for all $k$ and  such that 
\begin{equation}\label{exp:bubble:0}
B_\ve - \chi(d_g(\cdot,\xe ))\me^{-\frac{n-2}{2}}V\left(\frac{\varphi_\ve(\cdot)-\varphi_\ve(\xe)}{\me}\right) \to 0 \quad \text{ in } H^1(M) 
\end{equation}
as $\ve \to 0$. 
\end{defi}
Here $\Sigma$ is as in \eqref{defSigma}. As the following simple result shows, Definition \ref{def:bubble} is independent of the choice of the local chart up to an isometry. In the sequel, we let $O(n)$ be the group of isometries of the Euclidean space $\rn$:  
\begin{prop}\label{prop:unique:bubble} 
Let $(\xe )_{\ve>0}\in M$ and $(\me)_{\ve>0}\in \rr_{>0}$ be families of points in $M$ and positive real numbers, with $\mu_{\ve} \to  0$ as $\ve \to 0$, and let $B=(B_\ve)_\ve$ and $\bar B=(\bar B_\ve)_\ve$ be families of bubbles centered at $x_\ve$ with radius $\mu_\ve$ and respectively modeled on $V, \bar V \in \Sigma$ as in Definition \eqref{def:bubble}.  Then 
$$B_\ve -\bar B_\ve \to 0 \text{ in } H^1(M)  \text{ as } \ve \to 0 
\quad \iff \quad  
\bar{V}=V\circ O$$
for some orthogonal mapping $O\in O(n)$. 
\end{prop}
Families of bubbles $(B_\ve)_{\ve >0}$ and $(\bar B_{\ve})_{\ve >0}$ centered at $x_\ve$ with radius $\mu_\ve$ and satisfying $B_\ve -\bar B_\ve \to 0$  in $H^1(M)$ as $\ve \to 0$ will be called \emph{equivalent} families. 

\begin{proof}
Let $\varphi_\ve, \bar{\varphi}_\ve : B_g(x_0, r_g) \to \R^n$ be normal coordinate charts for $g$ at $\xe$ associated to $B$ and $\bar B$ as in definition \ref{def:bubble}. If $x \in \R^n$ we have
\begin{equation} \label{eq:normalcharts}
\varphi_\ve \circ \bar{\varphi}_\ve^{-1}( \bar{\varphi}_\ve(\xe)+x) = \varphi_\ve(\xe)+O_\ve x + O(|x|^2) \quad \text{ as } x \to 0, 
 \end{equation}
where we have let $O_\ve = D\varphi_\ve (\xe) \circ \big(D\bar{\varphi}_\ve(\xe) \big)^{-1}$. Since $ \varphi_\ve$ and $\bar{\varphi}_\ve$ are normal charts at $\xe$, $O_\ve \in O(n)$. Straightforward computations using \eqref{eq:decayV} and \eqref{eq:normalcharts} show that
\begin{equation} \label{eq:normalcharts2}
 \chi(d_g(\cdot,\xe )) \Bigg[ \me^{-\frac{n-2}{2}}V\left(\frac{\varphi_\ve(\cdot)-\varphi_\ve(\xe)}{\me}\right)  -
\me^{-\frac{n-2}{2}}V\left(\frac{O_\ve \left(\bar{\varphi}_\ve(\cdot)-\bar{\varphi}_\ve(\xe)\right) }{\me}\right)\Bigg]   \to 0 
\end{equation}
in $H^1(M)$ as $\ve \to 0$. Since $B_\ve - \bar B_\ve \to 0$ in $H^1(M)$ by assumption we obtain
$$ \chi(d_g(\cdot,\xe )) \Bigg[  
\me^{-\frac{n-2}{2}}V\left(\frac{O_\ve \left(\bar{\varphi}_\ve(\cdot)-\bar{\varphi}_\ve(\xe)\right) }{\me}\right) - \me^{-\frac{n-2}{2}}\bar V\left(\frac{ \bar{\varphi}_\ve(\cdot)-\bar{\varphi}_\ve(\xe)  }{\me}\right)  -\Bigg]   \to 0 $$
in $H^1(M)$ as $\ve \to 0$. A simple scaling argument then shows that $V \circ O = \bar V$ where $O:=\lim_{\ve\to 0}O_\ve$. 
\end{proof}
Equations \eqref{eq:normalcharts} and \eqref{eq:normalcharts2} show that, up to composing $V$ with an euclidean isometry, we can always assume that $\vp_{\ve}$ is an exponential chart at $x_\ve$ in Definition \ref{def:bubble}. Choose, for any $p \in B_g(x_0, r_g)$, an identification of $T_p M$ with $\R^n$. This defines an exponential map at $p$, that we will denote by $\exp_p: \R^n \to M$ in all of this paper. Mimicking the arguments in \eqref{eq:normalcharts} and \eqref{eq:normalcharts2} we similarly have, for $V \in \Sigma$,  
\begin{equation*}\chi(d_g(\cdot,\xe )) \Bigg[ \me^{-\frac{n-2}{2}}V\left(\frac{\varphi_\ve(\cdot)-\varphi_\ve(\xe)}{\me}\right)  - \me^{-\frac{n-2}{2}}\big( V \circ O \big) \left(\frac{\exp_{x_\ve}^{-1}(\cdot)}{\me}\right)\Bigg]   \to 0 
\end{equation*}
for some $O \in O(n)$. We have thus shown that $B = (B_{\ve})_{\ve >0}$ is a \emph{family of bubbles centered at $x_\ve$ with radius $\mu_\ve$ and modeled on $V\in \Sigma$} if and only if 
\begin{equation}\label{exp:bubble}
B_\ve - \chi(d_g(\cdot,\xe ))\me^{-\frac{n-2}{2}}\big( V \circ O \big)\left(\frac{\exp_{x_\ve}^{-1}(\cdot)}{\me}\right) \to 0 \quad \text{ in } H^1(M) 
\end{equation}
as $\ve \to 0$, for some $O \in O(n)$. Any two exponential charts at $x_\ve$ are related by an isometry of $\R^n$. A consequence of  \eqref{exp:bubble} is that Definition \ref{def:bubble} is independent of the choice of the exponential chart at $x_\ve$ up to composing $V$ with an isometry. The classical example of a family of bubbles is the so-called family of standard bubbles:
\begin{equation}\label{std:bubble}
\begin{aligned} 
B_\ve(x)& :=\chi(d_g(x,\xe ))B_{\ve}^+(x), \hbox{ where } \\
B_\ve^+(x)& :=\left(\frac{\me}{\me^2+\frac{d_g(x,\xe )^2}{n(n-2)}}\right)^{\frac{n-2}{2}} = \me^{-\frac{n-2}{2}}B_0^+\left(\frac{\exp_{x_\ve}^{-1}(x)}{\me}\right),
\end{aligned}
\end{equation}
where $B_0^+$ is given by \eqref{B0}. This example is peculiar: since $B_0^+$ is radially symmetrical (as a consequence of positivity), any exponential chart at $x_\ve$ yields the same family of bubbles, and the function $B_\ve^+$ thus obtained is well-defined on the whole of $M$. Naively composing with an exponential chart does not, however, allow to globally pull-back a sign-changing element $V \in \Sigma$ in $M$ without losing regularity at the cut-locus, hence the need for Definition \ref{def:bubble}.

\medskip

  The following proposition is straightforward:
\begin{prop}\label{lim:bubble} Let $(B_\ve)_{\ve>0}$ be a family of bubble as in Definition \ref{def:bubble}. Then 
\begin{itemize}
\item There exists $C>0$ such that $\Vert B_\ve\Vert_{H^1}\leq C$ for all $\ve >0$
\item $\lim_{\ve\to 0}\Vert B_\ve\Vert_{L^2}=0$
\item $B_\ve \rightharpoonup 0$ weakly in $H^1(M)$ as $\ve\to 0$ 
\item $\lim_{\ve\to 0}\Vert \nabla B_\ve\Vert_2^2=\int_{\rn}|\nabla V|^2\, dx$.
\end{itemize}
\end{prop}

\subsection{The product $\hbox{Weyl}_g\otimes B$}

We now investigate the properties of the term $ \hbox{Weyl}_g\otimes B$ appearing in the Statement of Theorem \ref{maintheo}. 
\begin{defi}\label{def:tensor:B} Assume that $n \ge 5$. Let $B=(B_\ve)_{\ve>0}$ be a family of bubbles centered at $x_\ve$ with radius $\mu_\ve$ and modeled on $V \in \Sigma$ as in Definition \ref{def:bubble}. 
We define 
\begin{equation} \label{eq:tensor:B}
\begin{aligned}
\hbox{Weyl}_g\otimes B&:= \frac{4n}{3(n-2)^2}  \left( \int_{\R^n} |V|^{\crit }dx \right)^{-1}\times \\
 &\hbox{Weyl}_g(x_0)_{i\alpha j\beta}    \int_{\rn}x^\alpha x^\beta \partial_{ij}^2 V \left(\frac{n-2}{2}V+x^l\partial_l V\right)\, dx
\end{aligned}
\end{equation}
where $x_0 = \lim_{\ve \to 0} x_\ve$ and where the coordinates of the Weyl tensor are taken with respect to the chart $\exp_{x_0}$. 
\end{defi}

It is intended in \eqref{eq:tensor:B} that repeated indices are summed over. That the second integral in \eqref{eq:tensor:B} is finite follows from \eqref{eq:decayV} since $n \ge 5$. Note that the term $\hbox{Weyl}_g \otimes B$ depends on $x_0$ since the family $B = (B_{\ve})_\ve >0$ does. We recall the symmetries of the Weyl tensor of $g$ at any point $x \in M$:
$$ \begin{aligned}
& \hbox{Weyl}_g(x)_{ijk\ell} = - \hbox{Weyl}_g(x)_{jik\ell} = - \hbox{Weyl}_g(x)_{ij \ell k} = \hbox{Weyl}_g(x)_{k\ell ij} \\
& \hbox{Weyl}_g(x)_{iji\ell} = 0 \text{ and }  \hbox{Weyl}_g(x)_{ijk\ell} + \hbox{Weyl}_g(x)_{jki\ell} + \hbox{Weyl}_g(x)_{kij\ell} = 0 \\
\end{aligned} $$
for any $i,j,k,\ell \in \{1, \dots, n\}$, where repeated indices are summed over. As a consequence of these symmetries, simple integration by parts (see for instance Hebey-Vaugon \cite{hvmz} or Lemma 3.2 in Mesmar-Robert \cite{mesmarrobert}) show that 
\begin{equation} \label{eq:tensor:B2}
\begin{aligned} 
\hbox{Weyl}_g\otimes B  = \frac{4n}{3(n-2)^2}  &\left( \int_{\R^n} |V|^{\crit }dx \right)^{-1} \times \\
& \hbox{Weyl}_g(x_0)_{i\alpha j\beta} \int_{\rn} x^\alpha x^\beta \partial_i V \partial_j V dx. \\
\end{aligned}
\end{equation}
Again the integral on the right is finite when $n \ge 5$ by \eqref{eq:decayV}. A consequence of \eqref{eq:tensor:B2} and of the tensorial nature of $\hbox{Weyl}_g(x_0)$ is that $\hbox{Weyl}_g\otimes B$ does not change if $V$ is replaced by $V \circ O$ for some $O \in O(n)$, provided the coordinates of $\hbox{Weyl}_g(x_0)$ are now computed in the chart $O \circ \exp_{x_0}^{-1}$. Hence $\hbox{Weyl}_g\otimes B$ does not depend on the choice of the exponential chart at $x_0$ and only depends on the equivalence class of families of bubbles at $x_0$ in the sense of Proposition \ref{prop:unique:bubble}.

We describe two simple cases where $\hbox{Weyl}_g \otimes B$ vanishes globally:

\begin{prop}\label{prop:WB} Assume that $n \ge 5$ and let $B$ be a family of bubbles centered at $(\xe )_\ve\in M$, with radius $(\me)_\ve\in \rr_{>0}$ and modeled on $V \in \Sigma$, and $x_0 = \lim_{\ve \to 0} x_\ve$. Then $\hbox{Weyl}_g\otimes B=0$ in the following (non-exhaustive) situations:
\begin{itemize}
\item $\hbox{Weyl}_g(x_0)=0$ 
\item $V$ is radial (and hence $V = B_0^+$ given by \eqref{B0}). 
\end{itemize}
\end{prop}
\begin{proof} By \eqref{eq:tensor:B} the first point is trivial. If $V = B_0^+$ is radial, the second point follows from the radiality of $B_0^+$, from the equivalent expression \eqref{eq:tensor:B2} and from the total antisymmetry of the Weyl tensor. 
\end{proof}

The next two results shows that $\hbox{Weyl}_g \otimes B$ cannot be expected to vanish in general:

\begin{prop}\label{prop:weyl:negative} Let $(\sph^p,g_p)$ and $(\sph^q, g_q)$ be the unit spheres of dimensions $p,q\geq 2$ endowed with their round metrics. Assume that $n=p+q\geq 5$ and that $V\in D^{1,2}(\rn)\backslash \{0\}$ solves \eqref{yamabe:intro} and is invariant under the action of $O(p)\times O(q)$ but not under the action of $O(n)$. Let $M=\sph^p\times \sph^q$ endowed with the metric $g=g_p\otimes g_q$. Let $B=(B_\ve)_{\ve >0}$ be a family of bubbles in $(M,g)$ modeled on $V$ as in Definition \ref{def:bubble}, centered at $(\xe )_{\ve>0}\in M$ and with radius $(\me)_{\ve>0}\in \rr_{>0}$ and $x_0 = \lim_{\ve \to 0} x_\ve$. Then $\hbox{Weyl}_g\otimes B<0$. 
\end{prop}

\begin{proof} We let $M:=\sph^p\times \sph^q$ where $p,q\geq 2$ and $p+q=n\geq 5$ and $g:=g_p\otimes g_q$ be the product metric of the round metrics $g_p$ and $g_q$ respectively on $\sph^p$ and $\sph^q$. We follow the geometric notations of Besse \cite{besse} and Hebey \cite{hebey.diderot}. The Riemann tensors for each metric are 
 $$\hbox{Rm}_{g_p}=\frac{1}{2}g_p\odot g_p\hbox{ and } \hbox{Rm}_{g_q}=\frac{1}{2}g_q\odot g_q,$$
 where $\odot$ is the Kulkarni-Nomizu product, given in coordinates by $(T\odot S)_{ijkl}=T_{ik}S_{jl}- T_{il}S_{jk}+T_{jl}S_{ik}-T_{jk}S_{il}$.  Standard results on product manifolds give that, after the necessary projections, $\hbox{Rm}_g = \frac{1}{2}g_p\odot g_p +\frac{1}{2}g_q\odot g_q$, $\hbox{Ric}_g = (p-1) g_p + (q-1) g_q$ and $S_g = p(p-1) + q(q-1)$. The Weyl tensor of $g$ is defined as
 $$\hbox{Weyl}_g=\hbox{Rm}_g-\frac{1}{n-2}\hbox{Ric}_g\odot g+\frac{S_g}{2(n-1)(n-2)}g\odot g.$$
If $x_0 = (x_0^p, x_0^q) \in \sph^p \times \sph^q$ we compute the coordinates of $\hbox{Weyl}_{g}(x_0)$ in the chart obtained from the exponential chart on $\sph^p$ at $x_0^p$ and the exponential chart on $\sph^q$ at $x_0^q$: this yields a normal chart on $\sph^p\times\sph^q$. The Latin letters $i,j,k...$ refer to coordinates in $\sph^p$, the Greek letter $\alpha,\beta,...$ refer to $\sph^q$. The previous observations show that, at any point around $x_0$ we have
 \begin{equation*}
 \begin{aligned}
 \big(\hbox{Weyl}_g\big)_{ijkl}&= C_1(g_{ik}g_{jl}-g_{il}g_{jk}) \\
 \big(\hbox{Weyl}_g\big)_{\alpha\beta\gamma\delta}& = C_2(g_{\alpha\gamma}g_{\beta\delta}-g_{\alpha\delta}g_{\beta\gamma}) \\
 \big(\hbox{Weyl}_g\big)_{i\alpha j\beta}&= -C_3 g_{ij}g_{\alpha\beta} \\
 \big(\hbox{Weyl}_g\big)_{ij\alpha\beta}&=0,
  \end{aligned}
  \end{equation*}
  with $C_1 =  \frac{ 2q(q-1)}{(n-1)(n-2) }, C_2 = \frac{ 2p(p-1)}{(n-1)(n-2) }$ and $C_3 = \frac{ 2(p-1)(q-1)}{(n-1)(n-2)}$. The mixed terms $\big(\hbox{Weyl}_g\big)_{ijk\alpha}$ and $\big(\hbox{Weyl}_g\big)_{i\alpha\beta\gamma}$ are null, and the other terms are obtained via the symmetries of the Weyl tensor. We denote the coordinates in $\R^n$ by $(X,Y)$, where $X \in \R^p$ and $Y \in \R^q$. Let $V \in \Sigma$ be $O(p) \times O(q)$ invariant and let $C_{n,V} =  \frac{4n}{3(n-2)^2}  \left( \int_{\R^n} |V|^{\crit }dx \right)^{-1}$. Using \eqref{eq:tensor:B2} and the explicit expression of $\hbox{Weyl}_g(x_0)$ we have
$$ \begin{aligned}
C_{n,V}^{-1}\hbox{Weyl}_g& \otimes B=C_1\int_{\rn}\left[|X|^2|\nabla_1 V|^2-(X,\nabla_1 V)^2\right]\, dXdY\\
&+C_2\int_{\rn}\left[|Y|^2|\nabla_2 V|^2-(Y,\nabla_2 V)^2\right]\, dXdY\\
&-C_3\int_{\rn}\Big[ |X|^2 |\nabla_2 V|^2 + |Y|^2 |\nabla_1 V|^2 - 2 (X, \nabla_1 V)(Y, \nabla_2 V) \Big]\, dX\, dY
\end{aligned} $$
where $\nabla_1 V $ and $\nabla_2 V$ denote respectively the gradient of $V$ with respect to $X \in \R^p$ and $Y \in \R^q$. 
The $O(p) \times O(q)$ invariance of $V$ implies that $V(X,Y)=V(|X|,|Y|)$, and as a consequence we have $\nabla_1 V(X,Y) = \partial_1V(X,Y) \frac{X}{|X|}$ and $\nabla_2 V(X,Y) = \partial_2V(X,Y) \frac{Y}{|Y|}$ where $\partial_1V$ and $\partial_2V$ are the partial derivatives of $(r_1,r_2)\mapsto V(r_1,r_2)$. We obtain in the end 
$$ \begin{aligned}
C_{n,V}^{-1}\hbox{Weyl}_g\otimes B&=-C_3\int_{\rn}\Big[ |X|^2 (\partial_2 V)^2+|Y|^2 (\partial_1 V)^2-2|X| |Y| \partial_1V \partial_2V\Big]\, dX\, dY\\
&=-C_3\int_{\rn}\Big( |X| \partial_2 V - |Y| \partial_1 V\Big)^2\, dX\, dY\leq 0.
\end{aligned}$$
If $\hbox{Weyl}_g\otimes B=0$, then $r_2 \partial_1V - r_1 \partial_2V$ for all   $r_1 = |X| >0$ and $r_2 =|Y| >0$. This shows that $V = V(r)$, for $r^2 = r_1^2+r_2^2$ when $r_1, r_2 >0$, and this equality extends continuously to $\rn \backslash \{0\}$. We then get that $V$ is radially symmetrical, contradicting our assumption. Therefore $\hbox{Weyl}_g\otimes B<0$.
  \end{proof}
  
  \begin{corol} \label{corol:weyl:negatif}
  There are manifolds $(M,g)$ of dimension $n\geq 5$ and families of bubbles $B=(B_\ve)_{\ve >0}$ as in Definition \ref{def:bubble} such that $\hbox{Weyl}_g\otimes B<0$.
\end{corol}
\begin{proof} We choose $p,q\geq 2$ such that $n:=p+q\geq 5$. As before we denote by $g_n$ the round metric in $\mathbb{S}^n$. It follows from the result of Ding \cite{ding} that for all $L>0$, there exists $\tilde V \in C^2(\mathbb{S}^n)$ that solves $\Delta_{g_n} \tilde V + \frac{n(n-2)}{4} \tilde V  = | \tilde V|^{2^*-2} \tilde V$,  satisfies $\int_{\mathbb{S}^n}| \tilde V|^{2^*}\, dv_{g_n}\geq L$ and is $O(p) \times O(q+1)$-invariant, that is $\tilde V(x,y)  = \tilde V(|x|,|y|)$ for $(x,y) \in \R^p\times\R^{q+1}=\R^{n+1}$. Let $\pi_N: \mathbb{S}^n \backslash \{N\} \to \R^n$ be the stereographic projection from the North pole. Let, for $x \in \R^n$,
$$ V(x) := \left(\frac{2}{1 + |x|^2}\right)^{\frac{n-2}{2}} \tilde V\big( \pi_N^{-1}(x) \big). $$
Then $V\in C^2(\rn)\cap D^{1,2}(\rn)$ solves $\Delta_\xi V=|V|^{\crit-2}V$ in $\R^n$, satisfies $\int_{\R^n} |V|^{2^*} \,dx \ge L$ and using the coordinate expression of the stereographic projection it is easily seen that it is $O(p)\times O(q)$-invariant, that is $V(x', x'') = V(|x'|, |x''|)$ for $(x',x'') \in \R^p \times \R^q = \R^n$. Up to taking $L$ large enough, we get that $V$ is not $O(n)$-invariant since by Obata \cite{Obata} finite-energy radial solutions of \eqref{yamabe:intro} are translations and rescalings of $B_0^+$ given by \eqref{B0} and all have fixed minimal Dirichlet energy. Let now $M:=\sph^p\times \sph^q$ and $g:=g_p\otimes g_q$ be the product metric of the round metrics on $\sph^p$ and $\sph^q$. Let $B=(B_\ve)_{\ve >0}$ be any family of bubbles in $(M,g)$ modeled on $V$ as in Definition \ref{def:bubble}, centered at $(\xe )_\ve\in M$ and with radius $(\me)_\ve\in \rr_{>0}$ and $x_0 = \lim_{\ve \to 0} x_\ve$, where $x_0 \in M$ is any point. For $L$ large enough Proposition \ref{prop:weyl:negative} applies and shows that $\hbox{Weyl}_g\otimes B< 0$. 
\end{proof}

  \subsection{The mass of the Green's function in 3D}

We conclude this section by recalling the definition of the mass when $n = 3$, following Li-Zhu \cite{LiZhu}:
\begin{defi}\label{def:mass} Let $(M,g)$ be a smooth, closed, connected $3-$dimensional Riemannian manifold. Let $h\in C^1(M)$ be such that $\Delta_g+h$ is coercive. Let $G_h$ be the Green's function for $\Delta_g+h$. Then for all $x_0\in M$, there exists $m_h(x_0)\in\mathbb{R}$, denoted as the \emph{mass}, such that
$$G_h(x,x_0)=\frac{1}{4\pi d_g(x,x_0)}+m_h(x_0)+o(1)\hbox{ as }x\to x_0.$$
\end{defi}
When $n=3$, $\Delta_g +h$ is coercive and $x_0 \in M$, It follows from the construction of the Green's function (see Aubin \cite{Aubin}, Robert \cite{RobDirichlet}) that there exists $\beta_{x_0}\in C^0(M)\cap H_{2 }^p(M)$ for all $p<3$ such that
\begin{equation} \label{DL:foncgreen}
G_h(x,x_0)=\frac{\chi(x)}{4\pi d_g(x,x_0)}+\beta_{x_0}(x)\hbox{ for all }x\in M \backslash \{x_0\},
\end{equation}
where $\chi \in C^\infty(M)$ is a cutoff function with $\chi(x)=1$ in $B_g(x_0, \frac{r_g}{2})$ and $\chi(x)=0$ in $M \backslash B_g(x_0, r_g)$. We have $(\Delta_g+h)\beta_{x_0}=-(\Delta_g+h)( (4\pi)^{-1}\chi d_g(\cdot,x_0)^{-1})=O(d_g(\cdot,x_0)^{-1})$ and standard elliptic theory then shows that 
\begin{equation} \label{est:C1:masse}
|\beta_{x_0}(x)|\leq C\, , \, |\nabla \beta_{x_0}(x)|\leq C\big (1+|\ln d_g(x_0,x)| \big) 
\hbox{ for all }x\in B_g\left(x_0, \frac{r_g}{2}\right)\backslash\{0\}.
\end{equation}
It now follows from Definition \ref{def:mass} that
$$m_{h}(x_0):=\beta_{x_0}(x_0).$$

\section{A priori pointwise estimates for families of bubbles} \label{Sec3}

\subsection{Statement of the results}

In this section we will assume, as in the introduction, that $(\ue)_{0<\ve\le1}$ is a family of possibly sign-changing solutions of \eqref{Intro:eq1} satisfying \eqref{one:bubble}: 
 \be u_\ve = c B_\ve + o(1) \quad \text{ in } H^1(M), 
\ee 
for some $c >0$, where $B_\ve$ is a family of bubbles as in Definition \ref{def:bubble}, centered at $x_\ve$ with radius $\mu_\ve$. Here $(x_\ve)_{0<\ve \le 1}$ and $(\mu_\ve)_{0 < \ve \le 1}$ are families of points in $M$ and of positive numbers with $\mu_\ve \to 0$ as $\ve \to 0$. We let $V\in \Sigma$, where $\Sigma$ is as in \eqref{defSigma}, be the function that appears in \eqref{exp:bubble}.  Under this assumption $(u_\ve)_{0 < \ve \le 1}$ blows-up with only one possibly sign-changing bubble and a zero weak limit.  Throughout this paper we will adopt the following conventions: for nonnegative families $(a_\ve)_{0<\ve\le 1}$ and $(b_\ve)_{0<\ve\le 1}$ of functions in $M$ (or real numbers), we will write $a_\ve \lesssim b_\ve $, or equivalently $a_\ve = O(b_\ve)$, when there exists a positive constant $C$ such that 
$$ a_\ve(x) \le C b_\ve(x) \quad \text{ for all } 0 < \ve \le 1 \text{ and any } x \in M.$$ 
Since $\Delta_g + h_\ve$ is coercive we let, for any $0 < \ve \le 1$, $G_{h_\ve}$ be the positive Green's function of $\Delta_g +h_\ve$ (see  Robert \cite{RobDirichlet}). Since $h_\ve \to h$ in $C^1(M)$ as $\ve \to 0$ and $\Delta_g +h$ is coercive there exists $C>0$ such that 
\ben \label{greenfonc}
\Big| d_g(x,y)^{n-2}G_{h_\ve}(x,y)- \frac{1}{(n-2) \omega_{n-1} }  \Big| \le C d_g(x,y)
\een holds true for any $(x,y) \in M\times M \backslash D(M)$ and for any $0 < \ve \le 1$, where $D(M) = \{ (x,x), x \in M \}$ and where $\omega_{n-1}$ is the area of the standard sphere $\mathbb{S}^{n-1}$. Define, for $ x\in M$, 
\[ F_\ve(x_\ve, x) =  (n-2)\omega_{n-1} d_g(x_\ve, x)^{n-2} G_{h_\ve}(x_\ve, x). \]
We define, for any $x \in M$ and any $0 < \ve \le 1$:
\ben \label{Ve2}
\bal
 B^{\mu_\ve,x_\ve}(x) & =  \chi \left(  d_{g}(x_\ve, x) \right)  F_\ve(x_\ve, x) \mu_\ve^{-\pui}V \left(  \frac{1}{\mu_\ve} \exp_{x_\ve}^{-1}(x) \right)    \\
 & + \left( 1 -  \chi \left(  d_{g}(x_\ve, x) \right)   \right)   (n-2)\omega_{n-1} \lambda(V) \mu_\ve^{\pui} G_{h_\ve}(x_\ve, x), \\
\eal \een
where $\lambda(V)$ is as in \eqref{deflambda} and for $\chi \in C^\infty_c(\R)$ with $\chi \equiv 1 $ in $[0,\frac{i_g(M)}{4}]$ and $\chi \equiv 0$ in $[\frac{i_g(M)}{2}, + \infty)$, and where $i_g(M)$ is the injectivity radius of $(M,g)$. By \eqref{greenfonc} and properties of the Green's function (see \cite{RobDirichlet}) we have 
$$ F_\ve(x_\ve, x) = 1 + O\big( d_g(x_\ve, x) \big) \quad \text{ and } \quad |\nabla_x F_\ve(x_\ve, x )| = O(1)  \quad \text{ for } x \in M.$$
It is therefore easily seen with \eqref{exp:bubble} that 
$$ B^{\mu_\ve,x_\ve} = B_\ve + o(1) \quad \text{ in } H^1(M), $$
so that $(B^{\mu_\ve, x_\ve})_{0 < \ve \le 1}$ is a family of bubbles which is equivalent to $(B_\ve)_{0< \ve \le 1}$ in the sense of Proposition \ref{prop:unique:bubble}. The original assumption \eqref{one:bubble} implies that $u_\ve$ still satisfies
\ben \label{Sec2:2}
 u_\ve = c B^{\mu_\ve,x_\ve} + o(1) \quad \text{ in } H^1(M)
\een
for some $c>0$. We will thus, in the subsequent analysis, work with the family $(B^{\mu_\ve, x_\ve})_{0 < \ve \le 1}$. If $(y_\ve)_{\ve}$ is a family of points in $M$ satisfying $\mu_\ve << d_g(x_\ve, y_\ve) \le \frac{i_g(M)}{4}$ we also have, by \eqref{eq:expansion}, \eqref{greenfonc} and \eqref{Ve2},
\begin{equation}\label{Ve3} \bal 
B^{\mu_\ve, x_\ve}(y_\ve) &= F_\ve(x_\ve, y_\ve) \mu_\ve^{-\pui}V \left(  \frac{1}{\mu_\ve} \exp_{x_\ve}^{-1}(y_\ve) \right)   \\
& = (n-2)\omega_{n-1} \mu_\ve^{\pui}G_{h_\ve}(x_\ve, y_\ve) \left(  \lambda(V) + O \left(  \frac{\mu_\ve}{d_g(x_\ve, y_\ve)}\right)   \right)  . 
\eal 
\end{equation}
The function $ B^{\mu_\ve,x_\ve}$ thus interpolates the (rescaled) pull-back of $V$ in $M$ to its first-order expansion at infinity. 

\medskip 

In this section we prove that \eqref{Sec2:2} still holds true globally pointwise in $M$, up to an error term that is of the order of the positive standard bubble $B_\ve^+$ given by \eqref{std:bubble}. The main result of this section is the following:

\begin{prop} \label{prop:C0}
Let $(M^n,g), n \ge 3$ be a smooth, connected and closed manifold, let $(h_\ve)_{0 < \ve \le 1}$ be a family of $C^1$ functions in $M$ converging in $C^1(M)$ towards $h$ and let $(p_\ve)_{0 < \ve \le 1}$ be a family of real numbers satisfying $2 < p_\ve \le \crit $ for all $0 < \ve \le 1$ and such that $\lim_{\ve \to 0} p_\ve = \crit $. Assume that $\Delta_g + h$ is coercive. Let $(u_\ve)_{0 < \ve \le 1}$ be a family of solutions of \eqref{Intro:eq1} satisfying \eqref{one:bubble} for some $c >0$. 

Then $c=1$ in \eqref{one:bubble} and there exists a family $(\sigma_\ve)_{0 < \ve \le 1}$ of positive real numbers with $\lim_{\ve \to 0} \sigma_\ve = 0$ such that, 
\ben \label{estPropC0}
 \left| \left| \frac{u_\ve - B^{\mu_\ve,x_\ve} }{B_\ve^+} \right| \right|_{L^\infty(M)} \le \sigma_\ve  
 \een
 for any $0< \ve \le 1$, where $B^{\mu_\ve,x_\ve}$ is as in \eqref{Ve2} and $B_\ve^+$ is as in \eqref{std:bubble}.
 \end{prop}
 
The function $V \in \Sigma$ that appears in the definition of $B^{\mu_\ve, x_\ve}$ is the one that defines the family $(B_{\ve})_{0 < \ve \le 1}$ and is given by \eqref{exp:bubble}. In particular, for any $x \in M$ and any $0 < \ve \le 1$ Proposition \ref{prop:C0} shows that
  $$ \big| u_\ve(x) - B^{\mu_\ve,x_\ve}(x) \big| \le \sigma_\ve B_\ve^+(x). $$
 As a consequence, if $(y_\ve)_{0<\ve\le 1}$ is any family of points in $M$ we have 
 $$ u_\ve(y_\ve) = B^{\mu_\ve,x_\ve}(y_\ve) + o\left(  B_\ve^+(y_\ve) \right)   $$
 as $\ve \to 0$. 
 
 \medskip
 
 Proposition \ref{prop:C0} is sharp if $\lambda(V) \neq 0$ which, by \eqref{eq:expansion}, corresponds to a bubble that decays to infinity in $\R^n$ as $|x|^{2-n}$. This is the case for the positive bubble $B_0^+$ given by \eqref{B0} and for the sign-changing solutions of \eqref{yamabe:intro} constructed in DelPino-Musso-Pacard-Pistoia \cite{DelPinoMussoPacardPistoia1}. Proposition \ref{prop:C0} for positive solutions has been known since the work of Li-Zhu \cite{LiZhu} and Druet \cite{DruetJDG} and Druet-Hebey-Robert \cite{DruetHebeyRobert}] (see also \cite{HebeyZLAM}).  In this section we prove Proposition \ref{prop:C0} by following the approach in Ghoussoub-Mazumdar-Robert \cite{GhoussoubMazumdarRobert}, that draws inspiration from the techniques in \cites{DruetHebeyRobert, HebeyZLAM}. Since we consider sign-changing solutions and bubbles we have to adapt the existing techniques that have been developed for positive solutions. In the exactly critical case $\pe \equiv \crit $, Proposition \ref{prop:C0} has already been proven in Premoselli \cite{Premoselli13} in the more general case where multiple sign-changing bubbles may appear.

 \subsection{Proof of Proposition \ref{prop:C0}}
 
As a preliminary observation we claim that  
\begin{equation}\label{defc}
 c = \lim_{\ve \to 0} \mu_\ve^{-\frac{(n-2)^2}{8}(\crit -\pe)} \in [1, + \infty).
 \end{equation}
We set $\tilde{U}_\ve(x):=\me^{\frac{n-2}{2}}\ue(\hbox{exp}_{\xe}(\me x))$ for $x\in B(0, \frac{i_g(M)}{2 \mu_\ve})$. It follows from \eqref{one:bubble} and \eqref{exp:bubble} that $\tilde{U}_\ve\to cV$ in $H^1_{loc}(\rn)$. With a change of variable, equation \eqref{Intro:eq1} writes
$$\Delta_{g_\ve} \tilde{U}_\ve+\me^2\he(\hbox{exp}_{\xe}(\me x))\tilde{U}_\ve=\me^{\frac{n-2}{2}(\crit-p_\ve)}|\tilde{U}_\ve|^{\crit-2}\tilde{U}_\ve\hbox{ in }B \Big(0, \frac{i_g(M)}{2 \mu_\ve}\Big), $$
where we have let $g_\ve(x) = \exp_{x_\ve}^*g (\mu_\ve x)$. 
Passing to the limit $\ve\to 0$ yields 
$$\Delta_{\xi}(cV)=\big( \lim_{\ve\to 0}\me^{\frac{n-2}{2}(\crit-p_\ve)} \big)c^{\crit-1}|V|^{\crit-2}V\hbox{ weakly in }\rn.$$
Since $\mu_\ve \to 0$ as $\ve \to 0$ and $p_\ve \le 2^*$ we have $\mu_\ve^{2^*-p_\ve} \le 1$  for $\ve$ small enough. Since $V \in D^{1,2}(\R^n) \backslash \{0\}$ it cannot be harmonic, which shows that $\lim_{\ve\to 0}\me^{\frac{n-2}{2}(\crit-p_\ve)} \in (0,1]$. Finally, $V$ solves \eqref{yamabe:intro}, which proves the claim.

\medskip\noindent As already remarked, we can assume that \eqref{Sec2:2} holds. This implies that the family $(u_\ve)_{0 < \ve \le 1}$ blows-up, that is 
\begin{equation*}
\lim_{\ve \to 0} \Vert u_\ve \Vert_{L^\infty(M)} = + \infty 
\end{equation*}
By  \eqref{defc} it is also easily seen that 
\ben \label{Sec2:3}
 \mu_\ve^{\pui - \frac{2}{p_\ve-2}} \to c \quad \text{ as } \ve \to 0.
 \een
Decomposition \eqref{Sec2:2} with \eqref{Ve2} shows in particular that $u_\ve \to 0$ in $H^1(M \backslash B_g(x_\ve, \delta))$  as $\ve \to 0$, for any $\delta >0$ fixed. The regularity theory of Trudinger \cite{Trudinger}, together with standard elliptic theory, then shows that
\ben \label{Sec2:31}
u_\ve \to 0 \quad \text{ in } \quad C^2_{loc}(M \backslash \{x_0\})
\een
as $\ve \to 0$, where we have let $x_0 = \lim_{\ve \to 0} x_\ve$. We prove Proposition \ref{prop:C0} in several steps. 

\medskip

\textbf{Step $1$:} Let, for any $x\in B(0, \frac{i_g(M)}{2 \mu_\ve})$,
$$ U_\ve(x) = \mu_\ve^{\frac{2}{p_\ve-2}} u_\ve \left(  \exp_{x_\ve}(\mu_\ve x) \right)  . $$
We claim that 
\ben \label{Sec2:4}
U_\ve \to V \quad \text{ in } C^{2}_{loc}(\R^n) 
\een
as $\ve \to 0$.
\begin{proof}[Proof of \eqref{Sec2:4}]
Let, for $x \in B(0, \frac{i_g(M)}{2 \mu_\ve})$, $g_\ve(x) = \exp_{x_\ve}^*g (\mu_\ve x)$. By \eqref{Intro:eq1} $U_\ve$ satisfies
\ben \label{Sec2:5}
 \Delta_{g_\ve} U_\ve + \mu_\ve^2 \tilde{h}_\ve U_\ve = |U_\ve|^{p_\ve-2} U_\ve \quad \text{ in } B\left( 0, \frac{i_g(M)}{2 \mu_\ve} \right)  , 
 \een
where we have let $ \tilde{h}_\ve(x) = h_\ve \left(  \exp_{x_\ve}(\mu_\ve x) \right)  $. Equation \eqref{Sec2:2} together with \eqref{Sec2:3} shows that, for any compact set $K \subset \subset \R^n$, we have 
$$ \bal \int_K |U_\ve - V|^{p_\ve} dv_{g_\ve} & = \mu_\ve^{\frac{2p_\ve}{p_\ve-2}-n} \int_{\exp_{x_\ve}(\mu_\ve K)} \Big| u_\ve - \mu_\ve^{\frac{n-2}{2} - \frac{2}{p_\ve-2}}B^{\mu_\ve,x_\ve} \Big|^{p_\ve} dv_g \\
& =  \mu_\ve^{\frac{2p_\ve}{p_\ve-2}-n} \int_{\exp_{x_\ve}(\mu_\ve K)} \Big| u_\ve - (c + o(1)) B^{\mu_\ve,x_\ve} \Big|^{p_\ve} dv_g \\
& \lesssim  \int_{\exp_{x_\ve}(\mu_\ve K)} \Big| u_\ve - (c+o(1)) B^{\mu_\ve,x_\ve} \Big|^{p_\ve} dv_g  \\
& = o(1)
\eal $$
 as $\ve \to 0$, where the second line follows from \eqref{Sec2:3} and in the third line we used that $\mu_\ve^{\frac{2p_\ve}{p_\ve-2}-n}  \lesssim 1$ since $\frac{2p_\ve}{p_\ve-2}-n \ge0$. The local regularity theory of Trudinger \cite{Trudinger} and standard elliptic theory using \eqref{Sec2:5} then show that the convergence of $U_\ve$ towards $V$ takes place in $C^2_{loc}(\R^n)$, which concludes the proof of \eqref{Sec2:4}. 
 \end{proof}
By \eqref{Sec2:2}, \eqref{Sec2:4} and Sobolev's inequality, we get 
 \ben \label{Sec2:32}
\lim_{R \to +\infty} \limsup_{\ve \to 0} \Vert u_\ve \Vert_{L^{\crit }(M \backslash B_g(x_\ve, R\mu_\ve))} = 0 . 
\een

\medskip

\textbf{Step $2$:} 
Let $0 < \delta < i_g(M)$ be fixed. We claim that
\ben \label{Sec2:6}
\lim_{R \to + \infty} \limsup_{\ve \to 0} \max_{B_g(x_\ve, \delta)\backslash B_g(x_\ve, R\mu_\ve)} d_g(x_\ve, x)^{\frac{2}{p_\ve-2}} |u_\ve(x)| = 0 .
\een
\begin{proof}
To prove \eqref{Sec2:6} we will prove  more generally that 
\ben \label{Sec2:7}
\limsup_{\ve \to 0} \max_{ x \in M} \left(  \mu_\ve + d_g(x_\ve, x)\right)  ^{\frac{2}{p_\ve-2}}\big| u_\ve - c B^{\mu_\ve,x_\ve}\big|(x) = 0. 
\een
It is easily seen using \eqref{Ve2} that \eqref{Sec2:6} follows from \eqref{Sec2:7}. To prove \eqref{Sec2:7} we proceed by contradiction and assume that, up to passing to a subsequence $(\ve_k)_{k \ge0}$ with $\lim_{k \to + \infty} \ve_k = 0$, there exists $\eta_0 >0$ and a sequence $(y_k)_{k \ge0}$ of points of $M$ such that
\ben \label{Sec2:8}
 \bal &\left(  \mu_{\ve_k} + d_g(x_{\ve_k}, y_{k})\right)  ^{\frac{2}{p_{\ve_k}-2}}\big| u_{\ve_k} - c B^{\mu_{\ve_k}, x_{\ve_k}}\big|(y_k) \\
&=  \max_{ x \in M} \left(  \mu_{\ve_k} + d_g(x_{\ve_k}, x)\right)  ^{\frac{2}{p_{\ve_k}-2}}\big| u_{\ve_k} - c B^{\mu_{\ve_k}, x_{\ve_k}}\big|(x)  \ge  \eta_0. \eal 
\een
By \eqref{Sec2:3} and \eqref{Sec2:4} we get that 
\ben \label{Sec2:81}
\frac{d_g(x_{\ve_k}, y_k)}{\mu_{\ve_k}} \to + \infty
\een 
as $k \to + \infty$. By \eqref{Ve2} and since $\frac{2}{p_{\ve_k}-2} - \pui \ge 0$ we thus have
\ben \label{Sec2:82}
\bal
& \left(  \mu_{\ve_k} + d_g(x_{\ve_k}, y_k)\right)  ^{\frac{2}{p_{\ve_k}-2}}\big| B^{\mu_{\ve_k}, x_{\ve_k}}\big|(y_k)  \\
&\quad  \lesssim \left( \frac{\mu_{\ve_k}}{\mu_\ve + d_g(x_{\ve_k}, y_k)}\right)^{\pui} \left( \mu_\ve + d_g(x_{\ve_k}, y_k)\right)  ^{\frac{2}{p_{\ve_k}-2} - \pui} \\
&\quad = o(1) 
\eal
\een
as $k \to + \infty$.  Using \eqref{Ve2} and \eqref{Sec2:31}, \eqref{Sec2:8} also shows that $d_g(x_{\ve_k}, y_k) \to 0$, and hence, with \eqref{Sec2:82}, that $|u_{\ve_k}(y_k)| \to + \infty$ as $k \to + \infty$. We let $\nu_k = |u_{\ve_k}(y_k)|^{- \frac{p_{\ve_k}-2}{2}}$, so that $\nu_k \to 0$, and for $x \in B(0, \frac{i_g(M)}{2 \nu_k})$ we let 
$$ w_k(x) = \nu_k^{\frac{2}{p_{\ve_k}-2}} u_{\ve_k} \left(  \exp_{y_k}(\nu_k x) \right)  . $$
By \eqref{Intro:eq1} $w_k$ satisfies
\ben \label{Sec2:83}
 \Delta_{g_k} w_k + \nu_k^2 \tilde{h}_k w_k = |w_k|^{p_{\ve_k}-2} w_k \quad \text{ in } B\left( 0, \frac{i_g(M)}{2 \nu_k} \right)  , 
 \een
where we have let $ \tilde{h}_k(x) = h_{\ve_k} \left(  \exp_{y_k}(\nu_k x) \right)  $ and $g_k = \exp_{y_k}^*g(\nu_k x)$. Independently, by  \eqref{Sec2:8}, \eqref{Sec2:81} and \eqref{Sec2:82} we see that 
\ben \label{Sec2:84} 
d_g(x_{\ve_k}, y_k) \gtrsim \nu_k.
\een With the latter, \eqref{Sec2:8} shows that there exists $r_0 >0$ small enough such that 
$$ \Vert w_k \Vert_{L^\infty(B(0,2r_0))} \le 2 $$
for any $k$ large enough. With \eqref{Sec2:83} and standard elliptic theory, and since $|w_k(0)|=1$, $w_k$ converges in  $C^2(\overline{B(0,r_0)})$ as $k \to + \infty$, up to a subsequence, to a non-zero solution $w_0$ of 
$$ \Delta_{\xi} w_0 = |w_0|^{\crit -2} w_0 \quad \text{ in } B(0,r_0) ,$$ 
that thus satisfies
\ben \label{Sec2:85}
\int_{B(0,r_0)} |w_0|^{\crit } dx >0. 
\een 
Independently, \eqref{Sec2:81} and \eqref{Sec2:84} show that, up to assuming that $r_0$ is small enough, we have
\ben \label{Sec2:86}
\frac{1}{\mu_{\ve_k}} d_g \left(  x_{\ve_k}, B_g(y_k, r_0 \nu_k) \right)   \to + \infty 
\een
as $k \to + \infty$. We let $q_k = \frac{n}{2} p_{\ve_k} -n$. As is easily seen, it satisfies $q_k \le \crit $ (since $p_{\ve_k} \le \crit $), $q_k \to \crit $ as $k \to + \infty$ and
$$ \frac{2q_k}{p_{\ve_k}-2} = n. $$
 We then have, by the $C^2$ convergence of $w_k$ towards $w_0$ and by H\"older's inequality,
$$\bal \int_{B(0,r_0)} |w_0|^{\crit } dx & =  \limsup_{k \to +\infty}  \int_{B(0, r_0)} |w_k|^{q_k} dv_{g_k} \\
 & = \limsup_{k \to +\infty}\int_{B_g(y_k, r_0 \nu_k)} |u_{\ve_k}|^{q_k} dv_{g}  \\
& \lesssim \limsup_{k \to +\infty}  \left(\int_{B_g(y_k, r_0 \nu_k)} |u_{\ve_k}|^{\crit } dv_g\right)^{\frac{q_k}{\crit }}\\
& = 0,
\eal$$
where the last equality follows from \eqref{Sec2:32} and \eqref{Sec2:86}. This is a contradiction with \eqref{Sec2:85}. Hence \eqref{Sec2:8} cannot occur, which proves \eqref{Sec2:7}.
\end{proof}

\medskip

For $0 <\delta < i_g(M)$ we let 
\ben \label{Sec2:9}
\eta_\ve(\delta) = \Vert u_\ve \Vert_{L^\infty(M \backslash B_g(x_\ve, \delta))}. 
\een 
By \eqref{Sec2:31}, $\eta_\ve(\delta) \to 0$ as $\ve \to 0$ for any $0 <\delta < i_g(M)$  fixed. 

\medskip

\textbf{Step $3$:} We claim that for any $0 < \alpha < \frac12$ there exists $0 < \delta_\alpha < i_g(M)$ and $C_{\alpha} >0$ such that, for any $x \in M \backslash B_g(x_\ve, \mu_\ve)$ and any $0 < \ve \le 1$
\begin{equation}\label{Sec2:10}
|u_\ve(x)| \le C_\alpha \left(  \mu_{\ve}^{\frac{2}{p_\ve-2}(1-2\alpha)} d_g(x_\ve, x)^{-\frac{4(1-\alpha)}{p_\ve-2}} + \eta_\ve(\delta_\alpha)  d_g(x_\ve, x)^{- \frac{4\alpha}{p_\ve-2}} \right)   
\end{equation}
holds. 

\begin{proof}[Proof of \eqref{Sec2:10}] Let $0 < \alpha < \frac12$. For $x \in M \backslash B_g(x_\ve, \mu_\ve)$ we let 
\[ \Phi_{\ve}^{\alpha}(x) = \mu_\ve^{\frac{2}{p_\ve-2}(1-2\alpha)} G_{h_\ve}(x_\ve, x)^{\frac{4(1-\alpha)}{(n-2)(p_\ve-2)}} + \eta_\ve(\delta)G_{h_\ve}(x_\ve,x)^{\frac{4\alpha}{(n-2)(p_\ve-2)}},\]
where $G_{h_\ve}$ is the Green's function of $\Delta_g+h_\ve$. It follows from properties of the Green's function that there exists $\delta_\alpha'>0$  such that
$$d_g(x, x_{\ve})^2\left(\frac{|\nabla_x G_{h_\ve}(x_{\ve},x)|^2}{G_{h_\ve}(x_{\ve},x)^2}-h_\ve(x)\right)\geq \frac{(n-2)^2}{4}\hbox{ for }d_g(x, x_{\ve})<\delta_\alpha'.$$
Let $\delta_\alpha\in (0, \delta_\alpha')$ be such that 
$$ 2\Vert h \Vert_{L^\infty(M)} \delta_\alpha^2 < (n-2)^2\alpha(1-\alpha) $$
and let $0 < \delta \le \delta_\alpha$ be fixed.  We claim that there exists $C >0$ such that, for any $0 < \ve \le 1$,
\ben \label{Sec2:101}
 \left| \left| \frac{u_\ve}{\Phi_{\ve}^\alpha} \right| \right|_{L^\infty(M \backslash B_g(x_{\ve}, \mu_\ve))}  \le C
 \een
holds. By \eqref{greenfonc} it is easily seen that \eqref{Sec2:10} follows from \eqref{Sec2:101}. We thus prove \eqref{Sec2:101}. We let, for any $0 < \ve \le 1$, $y_\ve \in M \backslash B_g(x_\ve, \mu_\ve)$ be such that
\[ \left| \frac{u_\ve(y_\ve)}{\Phi_{\ve}^{\alpha}(y_\ve)}\right| = \max_{x \in M \backslash B_g(x_\ve, \mu_\ve)} \left| \frac{u_\ve(x)}{\Phi_\ve^{\alpha}(x)} \right|.\]
We proceed by contradiction and assume that, up to passing to a subsequence $(\ve_k)_{k \ge0}$ with $\lim_{k \to + \infty} \ve_k = 0$, and up to replacing $u_{\ve_k}$ by $-u_{\ve_k}$ (which still solves \eqref{Intro:eq1}), we have 
\ben \label{Sec2:102}
 \frac{u_{\ve_k}(y_{\ve_k})}{\Phi_{\ve_k}^{\alpha}(y_{\ve_k})} \to + \infty 
 \een
 as $k \to + \infty$. We first assume that, up to a subsequence, 
\begin{equation} \label{Sec2:11}
\mu_{\ve_k} =o( d_g(y_{\ve_k}, x_{\ve_k})) \quad \text{ and } \quad d_g(x_{\ve_k}, y_{\ve_k}) \le \delta_\alpha
\end{equation}
 as $k \to +\infty$. By definition of $y_{\ve_k}$ we have
\begin{equation} \label{Sec2:12}
 \frac{\Delta_g u_{\ve_k}(y_{\ve_k})}{u_{\ve_k}(y_{\ve_k})} \ge \frac{\Delta_g \Phi_{\ve_k}^{\alpha}(y_{\ve_k})}{ \Phi_{\ve_k}^{\alpha}(y_{\ve_k})}.
 \end{equation}
On the one side, straightforward computations show that 
\[  d_g(x_{\ve_k}, y_{\ve_k})^2 \frac{\Delta_g \Phi_{\ve_k}^{\alpha}(y_{\ve_k})}{ \Phi_{\ve_k}^{\alpha}(y_{\ve_k})} \geq  \frac{(n-2)^2}{2}\alpha(1-\alpha) >0\hbox{ for }d_g(x, x_{\ve})<\delta_\alpha \]
as $k \to +\infty$. On the other side,   \eqref{Intro:eq1} and \eqref{Sec2:11} show that
\[ d_g(x_{\ve_k}, y_{\ve_k})^2  \frac{\Delta_g u_{\ve_k}(y_{\ve_k})}{u_{\ve_k}(y_{\ve_k})} \le  d_g(x_{\ve_k}, y_{\ve_k})^2 |u_{\ve_k}(y_{\ve_k})|^{p_{\ve_k}-2}  + 2\Vert h \Vert_{L^\infty(M)} \delta_\alpha^2 \]
as $k \to +\infty$. With \eqref{Sec2:6}, \eqref{Sec2:12} and the choice of $\delta_\alpha$ this yields a contradiction for $k$ large enough, and shows that \eqref{Sec2:11} cannot happen. Thus we either have, up to a subsequence, $\frac{d_g(x_{\ve_k}, y_{\ve_k})}{\mu_{\ve_k}} \to D \in [1, + \infty)$ or $d_g(x_{\ve_k}, y_{\ve_k}) \ge \delta_\alpha$ as $k \to + \infty$. Using \eqref{Sec2:4} in the first case and \eqref{Sec2:9} in the second case we obtain a contradiction with \eqref{Sec2:102}. This proves \eqref{Sec2:101} and thus \eqref{Sec2:10}.
\end{proof}

\medskip

\textbf{Step $4$:} We claim that  
\begin{equation} \label{Sec2:13}
\big| u_\ve(x) \big|\lesssim B_\ve^+(x) \quad \text{ for any } x \in M.
\end{equation}
where $B_\ve^+$ is defined in \eqref{std:bubble}.

\begin{proof}[Proof of \eqref{Sec2:13}]
We first prove a weaker statement. Let $0 < \delta < i_g(M)$ be fixed. We claim that there exists $C_\delta >0$ such that, for any $0 < \ve \le 1 $ and $x \in M$,
\begin{equation} \label{Sec2:131}
\big| u_\ve(x) \big|\le C_\delta \left( B_\ve^+(x) + \eta_\ve(\delta) \right)  
\end{equation}
holds, where $B_{\ve}^+$ is as in  \eqref{std:bubble}. Let $(y_\ve)_{0 < \ve \le 1}$ be any family of points in $M$. First, if $d_g(y_\ve, x_\ve) \ge \delta$, \eqref{Sec2:131} follows from \eqref{Sec2:9}. We can thus assume that $y_\ve \in B_g(x_\ve, \delta)$ for all $0 < \ve \le 1$. A representation formula for $u_\ve$ at $y_\ve$ gives, with \eqref{Intro:eq1}, 
\begin{equation} \label{Sec2:14}
\begin{aligned}
 u_\ve(y_\ve) & = \int_{B_g(x_\ve, \mu_\ve) } G_{h_\ve}(y_\ve, y) |u_\ve(y)|^{p_\ve-2}u_\ve(y) dy  \\
&+ \int_{M \backslash B_g(x_\ve, \mu_\ve)} G_{h_\ve}(y_\ve, y) |u_\ve(y)|^{p_\ve-2}u_\ve(y) dy.
\end{aligned} \end{equation}
Straightforward computations with \eqref{Sec2:3}, \eqref{Sec2:4} and \eqref{greenfonc} show that 
\begin{equation} \label{Sec2:15}
\left| \int_{B_g(x_\ve, \mu_\ve)} G_{h_\ve}(y_\ve, y) |u_\ve(y)|^{p_\ve-2}u_\ve(y) dy \right| \lesssim B_\ve^+(y_\ve)
\end{equation}
(see e.g. \cite[Proposition $6.1$]{HebeyZLAM}). Straightforward computations using \eqref{Sec2:3}, \eqref{Sec2:10} and \eqref{greenfonc} show that 
 \begin{equation} \label{Sec2:16}
\left| \int_{M \backslash B_g(x_\ve, \mu_\ve)} G_{h_\ve}(y_\ve, y) |u_\ve(y)|^{p_\ve-2}u_\ve(y) dy \right| \lesssim B_\ve^+(y_\ve) + \eta_\ve(\delta)^{p_\ve-1}.
\end{equation}
Plugging \eqref{Sec2:15} and \eqref{Sec2:16} in \eqref{Sec2:14} and since $\eta_\ve(\delta) \to 0$ as $\ve \to 0$ proves \eqref{Sec2:131}. We now claim that, for any $0 < \delta < i_g(M)$ fixed,
\ben \label{Sec2:132}
\eta_\ve(\delta) \lesssim \mu_\ve^{\pui}
\een  
holds for $\ve$ small enough. With \eqref{std:bubble}, \eqref{Sec2:131} and \eqref{Sec2:132} this will conclude the proof of \eqref{Sec2:13}. We prove \eqref{Sec2:132} by contradiction: we fix $0 < \delta < i_g(M)$ and, up to passing to a subsequence $(\ve_k)_{k \ge0}$ with $\lim_{k \to + \infty} \ve_k = 0$ we assume that 
\ben \label{Sec2:17}
\frac{\eta_{\ve_k}(\delta)}{\mu_{\ve_k}^{\pui}} \to + \infty
\een
as $k \to + \infty$. We let, for $k \ge 0$ and for any $x \in M$, 
$$\hat{u}_k (x) = \frac{u_{\ve_k}(x)}{\eta_{\ve_k}(\delta)},$$
and let $y_k \in M \backslash B_g(x_{\ve_k}, \delta)$ be such that $|\hat{u}_k(y_k)| = 1$. By \eqref{std:bubble}, \eqref{Sec2:131} and \eqref{Sec2:17} the family $(\hat{u}_k)_{k \ge0}$ is uniformly bounded in compact subsets of $M \backslash \{x_0\}$, where $x_0 = \lim_{\ve \to 0} x_\ve$. By \eqref{Intro:eq1}, standard elliptic theory and since $\eta_{\ve_k}(\delta)\to 0$ as $k \to + \infty$, $\hat{u}_k$ thus converges in $C^2_{loc}(M \backslash \{x_0\})$, up to a subsequence, to some function $\hat{u}_0 \in C^2(M \backslash \{x_0\})$ that satisfies $\Delta_g \hat{u}_0 + h \hat{u}_0 = 0$ in $M \backslash \{x_0\}$ and $|\hat{u}_0(y_0)| = 1$, where $y_0 = \lim_{k \to + \infty} y_k \in M \backslash B_g(x_0, \delta)$. Passing \eqref{Sec2:131} to the limit pointwise shows in addition that we have 
$$ |\hat{u}_0(x)| \le C_\delta \quad \text{ for all } x \in M \backslash \{x_0\} .$$
Classical arguments thus show that the singularity of $\hat{u}_0$ at $x_0$ is removable and that $\hat{u}_0$ satisfies $\Delta_g \hat{u}_0 + h \hat{u}_0 = 0$ in $M$ in a strong sense. This implies $\hat{u}_0 \equiv 0$  since $\Delta_g + h$ is coercive, which is a contradiction with $|\hat{u}_0(y_0)| = 1$. Thus \eqref{Sec2:132} holds true and \eqref{Sec2:13} is proven. 
\end{proof}

A consequence of \eqref{Sec2:13}, \eqref{Intro:eq1} and standard elliptic theory is the following estimate on the first and second derivatives of $u_\ve$: for $0 \le k \le 2$,
\ben \label{Sec2:18}
|\nabla^k u_\ve(x)| \lesssim \frac{\mu_\ve^{\pui}}{\left(  \mu_\ve + d_g(x_\ve, x) \right)  ^{n-2+k}} \quad \text{ for all } x \in M \text{ and } 0 < \ve \le 1.
\een
Estimate \eqref{Sec2:18} follows from a scaling argument and from standard elliptic theory.

\medskip

\textbf{Step $5$:} we now claim that 
\ben \label{Sec2:19}
\crit -p_\ve = \left \{ \bal & O(\mu_\ve)  & \text{ if } n = 3 \\ & O\left(\mu_\ve^2 \ln \frac{1}{\mu_\ve}\right) & \text{ if } n=4\\ & O(\mu_\ve^2) & \text{ if } n \ge 5 \eal \right. .
\een
as $\ve \to 0$. In particular $\lim_{\ve\to 0}\me^{\crit-p_\ve}=1$.

\begin{proof}[Proof of \eqref{Sec2:19}]
Let $0 < \delta < i_g(M)$ be fixed and let 
$$\Omega_\ve = B_g(x_\ve, \delta).$$ We write a Pohozaev identity for $u_\ve$ on $\Omega_\ve$: let $X$ be any smooth vector field in $M$. Integrating \eqref{Intro:eq1} by parts against $\langle X, \nabla u_\ve \rangle $ classically shows that
\ben \label{Sec2:20}
\bal
&\int_{\Omega_\ve} h_\ve u_\ve \langle X, \nabla u_\ve \rangle dv_g + \int_{\Omega_\ve} \left(  \nabla X - \frac{1}{2}  \textrm{div}_g X g \right)  (\nabla u_\ve, \nabla u_\ve) dv_g \\
&+ \int_{\partial \Omega_\ve} \left(  \frac12 \langle X, \nu \rangle |\nabla u_\ve|_g^2 - \langle X, \nabla u_\ve \rangle \partial_\nu u_\ve - \frac{1}{p_\ve} \langle X, \nu \rangle |u_\ve|^{p_\ve} \right)   d\sigma_g \\
& = - \frac{1}{p_\ve} \int_{\Omega_\ve} \textrm{div}_g X |u_\ve|^{p_\ve} dv_g
\eal
\een
(see e.g. \cite[Proposition 6.2]{HebeyZLAM}). Let now  $X_\ve$ be the smooth vector field whose coordinates in the exponential chart at $x_\ve$ are $(X_\ve(x))^i = x^i$. For $x \in \Omega_\ve$ and any $1 \le i,j \le n$ we have $\nabla_i X_\ve^j(x) = \delta_i^j + O(d_g(x_\ve, x)^2)$. As a consequence, 
\be 
 \int_{\Omega_\ve} \left(  \nabla X - \frac{1}{2}  \textrm{div}_g X g \right)  (\nabla u_\ve, \nabla u_\ve) dv_g = \int_{\Omega_\ve}\left(  1 - \frac{n}{2}  + O(d_g(x_\ve, \cdot)^2) \right)   |\nabla u_\ve|_g^2 dv_g
\ee
and
\be 
 - \frac{1}{p_\ve} \int_{\Omega_\ve} \textrm{div}_g X |u_\ve|^{p_\ve} dv_g = \int_{\Omega_\ve} \left( - \frac{n}{p_\ve} + O(d_g(x_\ve, \cdot)^2 \right)     |u_\ve|^{p_\ve} dv_g. 
\ee
Integrating \eqref{Intro:eq1} by parts shows independently that 
\be
\int_{\Omega_\ve} |\nabla u_\ve|_g^2 dv_g = \int_{\Omega_\ve} |u_\ve|^{p_\ve}dv_g - \int_{\Omega_\ve} h_\ve u_\ve^2 dv_g + \int_{\partial\Omega_\ve} u_\ve \partial_{\nu} u_{\ve} d\sigma_g,
\ee
so that \eqref{Sec2:20} becomes  
\begin{equation*} 
\begin{aligned}
 &n \left(  \frac{1}{\crit } - \frac{1}{p_\ve} \right)    \int_{\Omega_\ve} |u_\ve|^{p_\ve}dv_g \\
& =   \int_{\Omega_\ve} \left(  \frac{n-2}{2} u_\ve^2  + h_\ve u_\ve \langle X, \nabla u_\ve \rangle \right)   dv_g \nonumber\\
& + O \left(   \int_{\Omega_\ve} d_g(x_\ve, \cdot)^2 |\nabla u_\ve|_g^2 dv_g +  \int_{\Omega_\ve} d_g(x_\ve, \cdot)^2  |u_\ve|^{p_\ve} dv_g  \right)   \nonumber\\
&+ \int_{\partial \Omega_\ve} \left(  \frac12 \langle X, \nu \rangle |\nabla u_\ve|_g^2 - \langle X, \nabla u_\ve \rangle \partial_\nu u_\ve - \frac{n-2}{2} u_\ve \partial_{\nu} u_\ve \right.\nonumber\\
&\quad \quad \left.- \frac{1}{p_\ve} \langle X, \nu \rangle |u_\ve|^{p_\ve} \right)   d\sigma_g .\nonumber
\end{aligned}
\end{equation*}
Straightforward computations using \eqref{Sec2:13} and \eqref{Sec2:18} show that
\be 
\bal
& \int_{\Omega_\ve} d_g(x_\ve, \cdot)^2 |\nabla u_\ve|_g^2 dv_g +  \int_{\Omega_\ve} d_g(x_\ve, \cdot)^2  |u_\ve|^{p_\ve} dv_g \\
& +  \int_{\Omega_\ve} \big| h_\ve u_\ve \langle X, \nabla u_\ve \rangle \big|dv_g +  \int_{\Omega_\ve} |h_\ve| u_\ve^2 dv_g  = \left \{ \bal & O(\mu_\ve)  & \text{ if } n = 3 \\ & O(\mu_\ve^2 \ln \frac{1}{\mu_\ve}) & \text{ if } n=4\\ & O(\mu_\ve^2) & \text{ if } n \ge 5 \eal \right. .
\eal 
\ee
By \eqref{Sec2:13} and \eqref{Sec2:18} we have 
\be 
 \int_{\partial \Omega_\ve} \left(  \frac12 \langle X, \nu \rangle |\nabla u_\ve|_g^2 - \langle X, \nabla u_\ve \rangle \partial_\nu u_\ve  - \frac{n-2}{2} u_\ve \partial_{\nu} u_\ve- \frac{1}{p_\ve} \langle X, \nu \rangle |u_\ve|^{p_\ve} \right)   d\sigma_g = O(\mu_\ve^{n-2}). 
\ee
Plugging the latter computations in \eqref{Sec2:20} thus shows that
\ben \label{Sec2:21}
n \left(  \frac{1}{\crit } - \frac{1}{p_\ve} \right)    \int_{\Omega_\ve} |u_\ve|^{p_\ve}dv_g = \left \{ \bal & O(\mu_\ve)  & \text{ if } n = 3 \\ & O(\mu_\ve^2 \ln \frac{1}{\mu_\ve}) & \text{ if } n=4\\ & O(\mu_\ve^2) & \text{ if } n \ge 5 \eal \right. .
\een
Finally, using \eqref{Sec2:3} and \eqref{Sec2:4}, Fatou's lemma shows that 
$$ \int_{\Omega_\ve} |u_\ve|^{p_\ve}dv_g \ge c_0 $$
for some positive $c_0 >0$ independent of $\ve$. Going back to \eqref{Sec2:21} thus proves \eqref{Sec2:19}.
\end{proof}

\medskip

\textbf{Step $6$: End of the proof of Proposition \ref{prop:C0}.} We are now in position to conclude the proof of Proposition \ref{prop:C0}. The equality $c=1$ follows from \eqref{defc} and \eqref{Sec2:19} and we thus only need to prove  \eqref{estPropC0}. We proceed by contradiction and assume that \eqref{estPropC0} is false: up to passing to a subsequence $(\ve_k)_{k \ge0}$ with $\lim_{k \to + \infty} \ve_k = 0$, there exists a sequence $(y_k)_{k \ge 0}$ of points of $M$ such that 
\ben \label{Sec6:23}
|u_{\ve_k}(y_k) - B^{\mu_{\ve_k}, x_{\ve_k}}(y_k)| \ge \eta_0 B_{\ve_k}^+(y_k) 
\een
for some $\eta_0 >0$, where $B^{\mu_\ve, x_\ve}$ and $B_\ve^+$ are as in \eqref{Ve2} and \eqref{std:bubble}. Since by \eqref{Sec2:19} we have $c=1$, \eqref{Sec2:4} shows with \eqref{Sec6:23} that 
\ben \label{Sec2:24}
\frac{d_g(x_{\ve_k}, y_{\ve_k}) }{\mu_{\ve_k}} \to + \infty
\een
as $k \to + \infty$. We write a representation formula for $u_{\ve_k}$ in $M$: if $G_{h_{\ve_k}}$ again denotes the Green's function for $\Delta_g + h_{\ve_k}$, we have with \eqref{Intro:eq1}:
\ben \label{Sec2:25}
u_{\ve_k}(y_k) = \int_M G_{h_{\ve_k}}(y_k, y) |u_{\ve_k}(y)|^{p_{\ve_k}-2} u_{\ve_k}(y) dv_g(y). 
\een
Let $R>1$ be fixed and let $x \in B(0,R)$. By \eqref{greenfonc} and \eqref{Sec2:24} we have
\[ (n-2) \omega_{n-1} d_g(x_{\ve_k}, y_k )^{n-2} G_{h_{\ve_k}} \left( \exp_{x_{\ve_k}}(\mu_{\ve_k} x), y_{k}) \right)   \to 1 \quad \textrm{ if } d_g(x_{\ve_k}, y_{k}) \to 0 \]
uniformly in $B(0,R)$ as $k \to + \infty$. Since $h_{\ve_k} \to h$ in $C^1(M)$ we also have 
\[ G_{h_{\ve_k}} \left( y_k, \exp_{x_{\ve_k}}(\mu_{\ve_k} x) \right)   = G_h(x_{\ve_k}, y_k) + o(1) \quad \textrm{ if } d_g(x_{\ve_k}, y_k) \not \to 0\]
as $k \to +\infty$. 
Since $y_k$ satisfies \eqref{Sec2:24}, and by \eqref{greenfonc}, the dominated convergence theorem gives with the definition \eqref{Ve3}
\ben \label{theorie3} \begin{aligned}
&\int_{B_g(x_{\ve_k}, R\mu_{\ve_k})}G_{h_{\ve_k}}(y_k, \cdot) |u_{\ve_k}|^{p_{\ve_k}-2} u_{\ve_k} dv_g \\
& = (1+o(1)) G_{h_{\ve_k}}(x_{\ve_k}, y_k) \mu_{\ve_k}^{\pui} \int_{B(0,R)} |V|^{\crit -2} V dx \\
& = (1+o(1)) G_{h_{\ve_k}}(x_{\ve_k}, y_k)(n-2)\omega_{n-1} \mu_{\ve_k}^{\pui} \left(  \lambda(V)+ O(R^{-2}) \right)   \\
& =  B^{\mu_{\ve_k}, x_{\ve_k}}(y_k) + o(B_{\ve_k}(y_k))  + O(R^{-2} B_{\ve_k}(y_k))
\end{aligned} \een
where $\lambda(V)$ is as in \eqref{eq:expansion}, where the third line follows from \eqref{eq:decayV} and \eqref{deflambda} and the last one from \eqref{Ve2} and \eqref{Sec2:24}. Independently, straightforward computations show with \eqref{Sec2:13} that
\[ \Big| \int_{M \backslash B_g(x_{\ve_k}, R\mu_{\ve_k})}G_{h_{\ve_k}}(y_k, \cdot) |u_{\ve_k}|^{p_{\ve_k}-2} u_{\ve_k} dv_g  \Big| \lesssim R^{-2} B_{\ve_k}^+(y_k)+o\left(B_{\ve_k}^+(y_k)\right) \]
holds. Plugging the latter and \eqref{theorie3} in \eqref{Sec2:25} shows that, for any fixed $R >0$ and for any $k$ large enough we have 
\[ u_{\ve_k}(y_k) = B^{\mu_{\ve_k}, x_{\ve_k}}(y_k) + o\left( B_{\ve_k}^+(y_k) \right)   + O\left( R^{-2} B_{\ve_k}^+(y_k) \right)  . \]
By choosing $R$ large enough and passing to a subsequence we get a contradiction with \eqref{Sec6:23}. This concludes the proof of Proposition \ref{prop:C0}. $\square$

\medskip

We conclude this section by showing an improvement of \eqref{Sec2:18} when $k=1$ that will be needed in the next section. Let $(h_\ve)_{0 < \ve \le 1}$, $(p_\ve)_{0 < \ve \le 1}$ be as in the statement of Proposition \ref{prop:C0} and let $(u_\ve)_{0 < \ve \le 1}$ be a family of solutions of \eqref{Intro:eq1} satisfying \eqref{one:bubble} for some $c >0$. In particular Proposition \ref{prop:C0} applies to $(u_\ve)_{0<\ve \le 1}$ and \eqref{estPropC0} holds. We claim that there exists a family of positive real numbers $(\sigma_\ve)_{0< \ve \le 1}$, with $\sigma_\ve \to 0$  as $\ve \to 0$, and a positive constant $C$ such that for any $x \in M$ and any $0< \ve \le 1$ we have
\begin{equation} \label{est:der:better}
\Big|  \nabla u_\ve(x) - \nabla B^{\mu_\ve, x_\ve}(x) \Big|  \le \sigma_\ve \frac{\mu_{\ve}^{\frac{n-2}{2}}}{(\mu_\ve + d_g(x_\ve,x) )^{n-1}} + C \mu_\ve^{\frac{n-2}{2}}    
 \end{equation}
 where $B^{\mu_\ve, x_\ve}$ is as in \eqref{Ve2}. 

\begin{proof}[Proof of \eqref{est:der:better}]
Let $0<2 \delta < i_g(M)$ be fixed. We define $w_\ve(x) = u_\ve \big( \exp_{x_\ve}(x) \big)$ for $x \in B(0,2\delta)$. We first prove that there exists a family  $(\sigma_\ve)_{0< \ve \le 1}$, with $\sigma_\ve \to 0$  as $\ve \to 0$, and a constant $C>0$ such that for any $x \in B(0,2\delta)$ and any $0< \ve \le 1$ we have
\begin{equation} \label{est:der:better2}
\Big|  \partial_i w_\ve(x) - \mu_\ve^{- \frac{n}{2}} (\partial_i V)\Big( \frac{x}{\mu_\ve} \Big) \Big|  \le \sigma_\ve \frac{\mu_{\ve}^{\frac{n-2}{2}}}{(\mu_\ve + |x| )^{n-1}} + C \mu_\ve^{\frac{n-2}{2}}. 
 \end{equation}
Let $x \in B(0, 2 \delta)$. The representation formula \eqref{Sec2:25} can be differentiated with respect to $x$: this gives, for $1 \le i \le n$,
 $$\partial_i w_{\ve}(x)  = \int_M \partial_i G_{h_{\ve}} \big(  \exp_{x_\ve}(x), y\big) |u_{\ve}(y)|^{p_{\ve}-2} u_{\ve}(y) dv_g(y),$$
 where we have let $ \partial_i G_{h_{\ve}}(  \exp_{x_\ve}(x), y) = \frac{\partial}{\partial x_i} \big( G_{h_\ve} \big( \exp_{x_\ve}(\cdot), y \big) \big)(x)$. Standard results on the Green's functions (see e.g. \cite{RobDirichlet}) show that there exists $C >0$ such that for all $\ve$ small enough and all $x,y \in B(0, 2 \delta)$,
 \begin{equation} \label{greenfonc2}
  \Bigg| \partial_i G_{h_{\ve}} \big( \exp_{x_\ve}(x),  \exp_{x_\ve}(y) \big) + \frac{1}{\omega_{n-1}}\frac{(x-y)_i}{|x-y|^n}\Bigg| \le C|x-y|^{2-n} +C|x|\cdot |x-y|^{1-n}
  \end{equation}
holds. Let now, for $x \in \R^n$, $\tilde{V}_\ve(x) =  \mu_\ve^{-\pui}V \big(  \frac{x}{\mu_\ve}  \big)$. Let $(z_\ve)_{0 < \ve \le 1}$ be any family of points in $B(0,2\delta)$. Straightforward computations using \eqref{Sec2:19} and \eqref{estPropC0} show that 
 $$\begin{aligned} 
 \partial_i w_{\ve}(z_\ve)  & = -\int_{B(0,2 \delta)} \frac{1}{\omega_{n-1}}\frac{(z_\ve-y)_i}{|z_\ve-y|^n} |\tilde{V}_\ve|^{p_\ve-2} \tilde{V}_\ve dy + o\Bigg(\frac{\mu_{\ve}^{\frac{n-2}{2}}}{(\mu_\ve + |z_\ve| )^{n-1}}\Bigg) + O(\mu_\ve^{\frac{n-2}{2}}) \\
 & = -\int_{B(0,2 \delta)} \frac{1}{\omega_{n-1}}\frac{(z_\ve-y)_i}{|z_\ve-y|^n} |\tilde{V}_\ve|^{2^*-2} \tilde{V}_\ve dy + o\Bigg(\frac{\mu_{\ve}^{\frac{n-2}{2}}}{(\mu_\ve + |z_\ve| )^{n-1}}\Bigg) + O(\mu_\ve^{\frac{n-2}{2}}) \\
&+ \int_{B(0,2 \delta)} \frac{1}{\omega_{n-1}}\frac{(z_\ve-y)_i}{|z_\ve-y|^n} \left(|\tilde{V}_\ve|^{2^*-2}-|\tilde{V}_\ve|^{p_\ve-2}\right) \tilde{V}_\ve dy.
\end{aligned} $$
Let us estimate the last term. Using the control \eqref{defc} , we get that
$$\begin{aligned} 
 &\left| \int_{B(0,2 \delta)} \frac{1}{\omega_{n-1}}\frac{(z_\ve-y)_i}{|z_\ve-y|^n} \left(|\tilde{V}_\ve|^{2^*-2}-|\tilde{V}_\ve|^{p_\ve-2}\right) \tilde{V}_\ve dy\right| \\
 &\leq   \int_{B(0,2 \delta)}  |z_\ve-y|^{1-n} \left||\tilde{V}_\ve|^{2^*-p_\ve}-1\right| \cdot |\tilde{V}_\ve|^{p_\ve-1} dy\\
 &\leq   \int_{B(0,2 \delta)\cap \{|\tilde{V}_\ve|<\mu_\ve^{\frac{n-2}{2}}\}}  |z_\ve-y|^{1-n} \left||\tilde{V}_\ve|^{2^*-p_\ve}-1\right| \cdot |\tilde{V}_\ve|^{p_\ve-1} dy\\
 &+   \int_{B(0,2 \delta)\cap \{|\tilde{V}_\ve|\geq \mu_\ve^{\frac{n-2}{2}}\}}  |z_\ve-y|^{1-n} \left||\tilde{V}_\ve|^{2^*-p_\ve}-1\right| \cdot |\tilde{V}_\ve|^{p_\ve-1} dy\\
 &\leq  C\mu_\ve^{\frac{n-2}{2}(p_\ve-1)}+   \int_{B(0,2 \delta)\cap \{|\tilde{V}_\ve|\geq \mu_\ve^{\frac{n-2}{2}}\}}  |z_\ve-y|^{1-n} \left||\tilde{V}_\ve|^{2^*-p_\ve}-1\right| \cdot |\tilde{V}_\ve|^{p_\ve-1} dy\\
 \end{aligned} $$
The control \eqref{eq:decayV} and \eqref{defc} yield
$$\{|\tilde{V}_\ve|\geq \mu_\ve^{\frac{n-2}{2}}\}\Rightarrow  \left||\tilde{V}_\ve|^{2^*-p_\ve}-1\right|=o(1)$$
as $\ve\to 0$, so that we get 
$$\begin{aligned} 
 \partial_i w_{\ve}(z_\ve)
 & = -\int_{\R^n} \frac{1}{\omega_{n-1}}\frac{(z_\ve-y)_i}{|z_\ve-y|^n} |\tilde{V}_\ve|^{2^*-2} \tilde{V}_\ve dy + o\Bigg(\frac{\mu_{\ve}^{\frac{n-2}{2}}}{(\mu_\ve + |z_\ve| )^{n-1}}\Bigg) + O(\mu_\ve^{\frac{n-2}{2}}) \\
 & = \partial_i \tilde{V}_\ve(z_\ve) +  o\Bigg(\frac{\mu_{\ve}^{\frac{n-2}{2}}}{(\mu_\ve + |z_\ve| )^{n-1}}\Bigg) + O(\mu_\ve^{\frac{n-2}{2}}),
\end{aligned} $$
where the last equality follows from a representation formula for \eqref{yamabe:intro} in $\R^n$, since $V \in \Sigma$. This proves \eqref{est:der:better2}. Estimate \eqref{est:der:better} now follows from  \eqref{greenfonc}, \eqref{est:der:better2}, \eqref{greenfonc2} and the explicit expression of $B^{\mu_\ve, x_\ve}$ in \eqref{Ve2}.
\end{proof}

\begin{remark}
We assumed for simplicity here that $\Delta_g +h$ is coercive. If $\Delta_g +h$ had a kernel and, more generally, if $(u_\ve)_{0 < \ve \le 1}$ was a finite-energy family of solutions of \eqref{Intro:eq1} (without the one-bubble assumption \eqref{one:bubble}), an analogue of Proposition \ref{prop:C0} would still hold. We refer for instance to Premoselli \cite{Premoselli13}  where a generalisation of Proposition \ref{prop:C0} was proven when $p_\ve = \crit $ for all $0 < \ve \le 1$, in the general multi-bubble case and when $\Delta_g + h$ is allowed to have a kernel. The proof of \cite{Premoselli13} allows to deal with any multi-bubble configuration and would still adapt to the asymptotically critical setting $p_\ve \le \crit $, $p_\ve \to \crit $ as a careful inspection of the proof shows. In the one-bubble case \eqref{one:bubble} that we consider here, the proof of Proposition \ref{prop:C0} that we gave in this section is shorter and more direct. 
\end{remark}

 \section{Proof of Theorems \ref{maintheo} and \ref{maintheo2}} \label{Sec5}

Throughout this section we let $(M^n,g), n \ge 3$ be a smooth, connected and closed manifold, $(h_\ve)_{0 < \ve \le 1}$ be a family of $C^1$ functions in $M$ converging in $C^1(M)$ towards $h$ and let $(p_\ve)_{0 < \ve \le 1}$ be a family of real numbers satisfying $2 < p_\ve \le \crit $ for all $0 < \ve \le 1$ and such that $\lim_{\ve \to 0} p_\ve = \crit $. We assume that $\Delta_g + h$ is coercive. We assume again, as in Section \ref{Sec3}, that $(\ue)_{0<\ve\le1}$ is a family of possibly sign-changing solutions of \eqref{Intro:eq1} satisfying \eqref{one:bubble}: 
 \be u_\ve = c B_\ve + o(1) \quad \text{ in } H^1(M), 
\ee 
for some $c >0$, where $B_\ve$ is a family of bubbles as in Definition \ref{def:bubble} centered at $x_\ve$ with radius $\mu_\ve$ and modeled on $V \in \Sigma$. In particular the analysis of Section \ref{Sec3} applies: Proposition \ref{prop:C0}, \eqref{Sec2:18} and \eqref{est:der:better} show that there exists a family $(\sigma_\ve)_{0 < \ve \le 1}$ with $\sigma_\ve \to 0$ as $\ve\to 0$ and a positive constant $C$ such that, for any $x \in M$ and any $\ve >0$,
 \ben \label{estimees} 
  \bal 
   \big| u_\ve(x) - B^{\mu_\ve,x_\ve}(x) \big| & \le \sigma_\ve B_\ve^+(x) , \\
   \Big|  \nabla u_\ve(x) - \nabla B^{\mu_\ve, x_\ve}(x) \Big|  &\le \sigma_\ve \frac{\mu_{\ve}^{\frac{n-2}{2}}}{(\mu_\ve + d_g(x_\ve,x) )^{n-1}} + C \mu_\ve^{\frac{n-2}{2}}   \quad \text{ and } \\
  \left( \mu_\ve + d_g(x_\ve, x) \right)  ^k|\nabla^k u_\ve(x)| &\le C B_\ve^+(x) \quad \text{ for } k = 1,2
\eal  
 \een
hold, where  $B_\ve^+$ and $B^{\mu_\ve,x_\ve}$ are as in \eqref{std:bubble} and \eqref{Ve2}. Equation \eqref{Sec2:19} also applies and shows that $\crit -p_\ve = O(\mu_\ve)$. In this section we will prove Theorem \ref{maintheo} using \eqref{estimees}. 
\medskip We first recall the classical conformal normal coordinates result of Lee-Parker \cite{LeeParker}. It states the existence of a positive function $\Lambda \in C^\infty(M \times M)$ such that, letting $\Lambda_y(x) = \Lambda(y,x)$ for any $y,x \in M$ and $g_{y}: = \Lambda_y^{\frac{4}{n-2}} g$ we have
\ben \label{propLambda}
 \bal 
 \Lambda_y(x) & = 1 + O \left(  d_{g_y}(y, x)^2 \right)  , \\
 \text{Ric}_{g_y}(x) & = O \left(  d_{g_y}(y, x) \right)   \\
 S_{g_y} & = O \left(  d_{g_y}(y, x)^2 \right)   \quad \text{ and } \\
 \sqrt{|g_y|}(x) & = 1 + O\left(d_{g_y}(y, x)\right)  ^N  
\eal 
\een
as $d_{g_y}(y,x) \to 0$, for a fixed $N$ large enough, where $\text{Ric}_{g_y}$ denotes the Ricci tensor of $g_y$ and where it is understood that $\sqrt{|g_y|}$ is computed with respect to the exponential map of $g_{y}$ at $y$. For $\ve >0$ we let $g_{x_\ve} =  \Lambda_{x_\ve}^{\frac{4}{n-2}} g$ and we let $\exp_{x_\ve}^{g_{x_\ve}}$ be the exponential chart of $g_{x_\ve}$ at $x_\ve$. Let $\delta >0$ be such that $2 \delta <   \inf_{\xi\in M} i_{g_\xi}(M)$. We define $g_\ve$ and $v_\ve$ in $B(0,2\delta) \subset \R^n$ as follows:
\begin{equation}\label{def:ve}
 \bal g_\ve & =  \left( \exp_{x_\ve}^{g_{x_\ve}}\right)  ^* g_{x_\ve} \quad \text{ and } \\
v_\ve(x) & = \frac{u_\ve}{\Lambda_{x_\ve}} \circ  \exp_{x_\ve}^{g_{x_\ve}} .
\eal 
\end{equation}
First, \eqref{propLambda} shows that, for any fixed $x \in M$, we have 
\begin{equation} \label{DLmetriqueconf} d_{g_{x}}(x, y) = d_g(x, y) + O \left(  d_g(x, y)^3 \right)  \text{ as } d_g(x,y) \to 0. 
\end{equation}
The latter with \eqref{eq:decayV}, \eqref{Ve2} and \eqref{estimees} implies that 
 \ben \label{estimees2} 
  \bal 
   \Big| v_\ve(x) - \tilde{V}_\ve (x) \big| & \le \sigma_\ve \tilde{B}_\ve^+(x) +  C\mu_\ve^{\pui}, \\
   \Big|  \nabla v_\ve(x) - \nabla \tilde{V}_\ve(x) \Big| & \le \sigma_\ve \frac{\mu_{\ve}^{\frac{n-2}{2}}}{(\mu_\ve + |x| )^{n-1}} +C \mu_\ve^{\frac{n-2}{2}}  \quad \text{ and } \\
  \left( \mu_\ve + d_{g_\ve}(x_\ve, x) \right)  ^k|\nabla^k v_\ve(x)| &\le C \tilde{B}_\ve^+(x) \quad \text{ for } k = 1,2
\eal  
 \een
 for any $x \in B(0,2\delta)$, where $C>0$ is independent of $\ve$ and where we have let
 $$ \tilde{B}_\ve^+(x) = \frac{\mu_\ve^{\frac{n-2}{2}}}{\left(  \mu_\ve^2 + \frac{|x|^2}{n(n-2)}\right)  ^{\frac{n-2}{2}}} \quad \text{ and } \quad \tilde{V}_\ve(x) =  \mu_\ve^{-\pui}V \left(  \frac{x}{\mu_\ve}  \right)  . $$
Moreover, it follows from  \eqref{Sec2:4} that
\begin{equation}\label{cv:hve}
\lim_{\ve\to 0}\hat{v}_\ve=V\hbox{ in }C^2_{loc}(\rn), \hbox{ where }\hat{v}_\ve(x):=\me^{\frac{n-2}{2}}v_\ve(\me x).
\end{equation}
Using \eqref{Intro:eq1} and the conformal invariance property of the conformal Laplacian it is easily seen that $v_\ve$ satisfies 
\ben \label{eq:neweq}
\bal \Delta_\xi v_\ve + A_\ve(v_\ve) + \hat{h}_\ve v_\ve= \Lambda_{\ve}^{p_\ve - \crit } |v_{\ve}|^{p_\ve-2} v_\ve \quad \text{ in } B(0, 2 \delta),
\eal 
\een
where 
$\Lambda_\ve:=\Lambda_{x_\ve}\circ \exp_{x_\ve}^{g_{x_\ve}}$,  $\xi$ denotes the Euclidean metric and where we have let $\Delta_\xi = - \sum_{i=1}^n \partial_i^2$,
\begin{equation}\label{def:hat:h}
 \hat{h}_\ve(x ) = \Bigg[\Lambda_{x_\ve}^{2 - \crit }\left(  h_\ve - \frac{n-2}{4(n-1)} S_g \right)  \Bigg] \left(  \exp_{x_\ve}^{g_{x_\ve}}(x)\right)   
 \end{equation}
and 
\begin{equation}\label{def:Ae}
 A_\ve(v_\ve) =  \left(  \Delta_{g_\ve} - \Delta_\xi\right)   v_\ve + \frac{n-2}{4(n-1)} S_{g_\ve} v_\ve  .
 \end{equation}
Integrating \eqref{eq:neweq} against $\langle x, \nabla v_\ve\rangle + \frac{n-2}{2} v_\ve$ in $B(0, \delta)$ with respect to the Euclidean metric yields the following Pohozaev identity (see Marques \cite[Lemma 2.1]{Marques}):
\ben \label{pohofinal1}
\bal \int_{\partial B(0,\delta)} &\left(  \frac{n-2}{2} v_\ve \partial_{\nu} v_\ve - \frac{\delta}{2} |\nabla v_\ve|_\xi^2 + \delta (\partial_\nu v\ve)^2 + \frac{1}{p_\ve} \Lambda_{\ve}^{p_\ve-\crit } \delta |v_\ve|^{p_\ve} \right)   d\sigma \\
&= \frac{1}{p_\ve} \int_{B(0,\delta)} (p_\ve-\crit ) x^k \partial_k \Lambda_{\ve} \Lambda_{x_\ve}^{p_\ve-1-\crit } |v_\ve|^{p_\ve}dx \\
&+ n \left(  \frac{1}{p_\ve} - \frac{1}{\crit } \right)  \int_{B(0,\delta)} \Lambda_{\ve}^{p_\ve-\crit } |v_\ve|^{p_\ve}dx \\
&+ \int_{B(0,\delta)} \left(  x^k \partial_k v_\ve + \frac{n-2}{2} v_\ve\right)   A_\ve(v_\ve) dx \\
& +  \int_{B(0,\delta)} \left(  x^k \partial_k v_\ve + \frac{n-2}{2} v_\ve \right)   \hat{h}_\ve v_\ve dx.
\eal
\een
In all the computations of this section it is intended that all the terms involving $p_\ve-\crit $ disappear in the case where $p_\ve = \crit $ for any $\ve >0$. We now estimate all the integrals appearing in \eqref{pohofinal1}.  First, straightforward computations using \eqref{propLambda} and \eqref{estimees2} show that 
\ben \label{pohofinal2}
 \bal 
\frac{1}{p_\ve} &\int_{B(0,\delta)} (p_\ve-\crit ) x^k \partial_k \Lambda_{\ve} \Lambda_{\ve}^{p_\ve-1-\crit } |v_\ve|^{p_\ve}dx \\
&+ n \left(  \frac{1}{p_\ve} - \frac{1}{\crit } \right)  \int_{B(0,\delta)} \Lambda_{\ve}^{p_\ve-\crit } |v_\ve|^{p_\ve}dx \\
& = \frac{(n-2)^2}{4n}   (\crit - p_\ve)\int_{\R^n} |V|^{\crit }dx + o\left(  |p_\ve-\crit |\right)  
\eal \een
and that 
\ben \label{pohofinal21}
\bal
\int_{\partial B(0,\delta)} \left(  \frac{n-2}{2} v_\ve \partial_{\nu} v_\ve - \frac{\delta}{2} |\nabla v_\ve|_\xi^2 + \delta (\partial_\nu v_\ve)^2 + \frac{1}{p_\ve} \Lambda_{\ve}^{p_\ve-\crit } \delta |v_\ve|^{p_\ve} \right)   d\sigma \\
= O(\mu_\ve^{n-2}) = \left\{ \begin{array}{cc}
o(\mu_\ve^2) &\hbox{ if }n\geq 5\\
O(\mu_\ve^2) &\hbox{ if }n=4\\
\end{array}\right.
\eal 
\een
as $\ve \to 0$. A simple integration by parts shows that
\begin{equation}\label{poho:35} \bal
&  \int_{B(0,\delta)} \left(  x^k \partial_k v_\ve + \frac{n-2}{2} v_\ve \right)   \hat{h}_\ve v_\ve dx \\
& = \frac12 \int_{\partial B(0,\delta)} \delta \hat{h}_\ve v_\ve^2d\sigma  - \int_{B(0,\delta)} \left(  \hat{h}_\ve + \frac12 x^k \partial_k \hat{h}_\ve \right)   v_\ve^2 dx,
\eal 
\end{equation}
so that with \eqref{estimees2}, we get that

\begin{equation}\label{poho:78} 
\begin{aligned}
 \int_{B(0,\delta)} &\left( x^k \partial_k v_\ve + \frac{n-2}{2} v_\ve \right) \hat{h}_\ve v_\ve dx  \\
 & =   - \int_{B(0,\delta)} \left( \hat{h}_\ve + \frac12 x^k \partial_k \hat{h}_\ve \right) v_\ve^2 dx +\left\{ \begin{array}{cc}
o(\mu_\ve^2) &\hbox{ if }n\geq 5\\
O(\mu_\ve^2) &\hbox{ if }n=4\\
\end{array}\right.
 \\
 & =  - \int_{B(0,\delta)} \hat{h}_\ve v_\ve^2 dx+\left\{ \begin{array}{cc}
o(\mu_\ve^2) &\hbox{ if }n\geq 5\\
O(\mu_\ve^2) &\hbox{ if }n=4\\
\end{array}\right.
\end{aligned}
\end{equation}
Let now $w \in C^\infty(B(0,2\delta))$ be a smooth function. We have, for any $\ve >0$ and by \eqref{propLambda},
$$ \Delta_{g_\ve} w - \Delta_{\xi} w = - \left( g_\ve^{ij} -\delta^{ij}\right)   \partial^2_{ij}w - \partial_i (g_\ve)^{ij} \partial_j w + O\left( |x|^{N-1} |\nabla w|_{\xi} \right)  $$
where repeated indices are summed over. By \eqref{propLambda} we have $\text{Ric}_{g_\ve}(x_\ve) = 0$. Cartan's expansion of $g_\ve$ and the symmetries of the Riemann tensor show that for any $1 \le j \le n$ 
$$  \partial_i (g_\ve)^{ij}(x ) =  - \frac13 \text{Ric}_{g_\ve}(x_\ve)_{pj} x^p + O(|x|^2)= O(|x|^2). $$
By \eqref{propLambda} again we also have $S_g(x_\ve) = 0$, so that $\text{Rm}_{g_\ve}(x_\ve) = \hbox{Weyl}_{g_\ve}(x_\ve)$, where $\hbox{Weyl}_{g_\ve}$ is the $(4,0)$ Weyl tensor of $g_\ve$. Using again Cartan's expansion of $g_\ve$ we find that for any smooth function $w$ in $B(0, 2 \delta)$, we have
\begin{equation}\label{est:diff}
 \Delta_{g_\ve} w - \Delta_{\xi} w =  \frac13 \hbox{Weyl}_{g_\ve}(x_\ve)_{ipqj}x^p x^q \partial^2_{ij} w + O\left( |x|^3 |\nabla^2 w| \right)   + O\left( |x|^2 |\nabla w| \right)  .
 \end{equation}
Applying the latter to $v_\ve$ and using \eqref{estimees2} shows that
$$A_\ve(v_\ve) =  \frac13 \hbox{Weyl}_{g_\ve}(x_\ve)_{ipqj}x^p x^q \partial^2_{ij} v_\ve + O \Bigg( \frac{\mu_\ve^{\frac{n-2}{2}}}{(\mu_\ve + |x|)^{n-3}} \Bigg). $$
Integrating the latter against $x^k \partial_k v_\ve + \frac{n-2}{2} v_\ve$ thus gives, using again \eqref{estimees2}, that 
\begin{equation} \label{term:weyl}
\begin{aligned}
&\int_{B(0,\delta)} \left(  x^k \partial_k v_\ve + \frac{n-2}{2} v_\ve\right)   A_\ve(v_\ve) dx \\
&=  \frac13 \hbox{Weyl}_{g_\ve}(x_\ve)_{ipqj} \int_{B(0,\delta)} x^p x^q \partial^2_{ij} v_\ve \left(  x^k \partial_k v_\ve + \frac{n-2}{2} v_\ve\right)  dx\\
& +\left\{ \begin{array}{cc}
o(\mu_\ve^2) &\hbox{ if }n\geq 5\\
O(\mu_\ve^2) &\hbox{ if }n=4\\
\end{array}\right.
\end{aligned}
\end{equation}
where again repeated indices are summed over. We now separate between the cases $n=3$, $n=4$ and $n\geq 5$ and prove Theorems \ref{maintheo} and \ref{maintheo2}.

\subsection{The case $n\geq 5$.} With $\hat{v}_\ve$ and $\hat{h}_\ve$ as in \eqref{cv:hve} and \eqref{def:hat:h} a change of variables gives
$$\int_{B(0,\delta)} \hat{h}_\ve dx =\me^2\int_{B(0,\me^{-1}\delta)} \hat{h}_\ve(\me x) \hat{v}_\ve^2 dx.$$
It follows from \eqref{estimees2} that $|\hat{v}_\ve(x)|\leq C(1+|x|)^{2-n}\in L^2(\rn)$ for $n\geq 5$. Therefore, using \eqref{cv:hve} and Lebesgue's convergence theorem, we get that \eqref{poho:78} yields
\ben \label{pohofinal3}
\bal 
&\int_{B(0,\delta)} \left(  x^k \partial_k v_\ve + \frac{n-2}{2} v_\ve \right)   \hat{h}_\ve v_\ve dx \\
 & \quad = - \left(  h(x_0) - \frac{n-2}{4(n-1)} S_g(x_0) \right)   \int_{\R^n} V^2dx \cdot \mu_\ve^2 +o(\mu_\ve^2) 
\eal \een
as $\ve \to 0$, where $x_0 = \lim_{\ve \to 0} x_\ve$. This limit exists up to passing to a subsequence as $\ve \to 0$, which will be implicit from now on. Note that the integral $\int_{\R^n} V^2dx$ is finite by \eqref{eq:decayV} since $n \ge 5$.

\smallskip\noindent We now deal with \eqref{term:weyl}, still for $n\geq 5$. By construction $g_\ve$ converges towards $\left( \exp_{x_0}^{g_{x_0}} \right)  ^* g_{x_0}$ in $C^1(\overline{B(0,\delta)})$ as $\ve \to 0$, so that $ \hbox{Weyl}_{g_\ve}(x_\ve)_{ipqj} \to \hbox{Weyl}_{g_{x_0}}(x_0)_{ipqj}$ as $\ve \to 0$, where the components of $\hbox{Weyl}_{g_{x_0}}$ are measured in the exponential chart of $g_{x_0}$ at $x_0$. Since $g_{x_0}$ is conformal to $g$ and since $\Lambda_{x_0}(x_0) = 1$ we have $g_{x_0}(x_0) = g(x_0)$, and by the conformal invariance of the Weyl tensor we also have 
$$ \hbox{Weyl}_{g_{x_0}}(x_0)_{ipqj} =  \hbox{Weyl}_{g}(x_0)_{ipqj},$$
where the components of  $\hbox{Weyl}_{g}(x_0)$ in the right-hand side are measured in the exponential chart of $g$ at $x_0$. By \eqref{estimees2} and the dominated convergence theorem we therefore obtain, with \eqref{term:weyl}, that 
 \ben  \label{pohofinal4}
\bal
& \int_{B(0,\delta)} \left(  x^k \partial_k v_\ve + \frac{n-2}{2} v_\ve\right)   A_\ve(v_\ve) dx \\
 & = \frac13 \hbox{Weyl}_{g}(x_0)_{ipqj} \int_{\R^n}  x^p x^q \partial^2_{ij} V\left(  x^k \partial_k V+ \frac{n-2}{2} V \right)  dx \cdot \mu_\ve^2 + o( \mu_\ve^2) \\
 & = - \frac{(n-2)^2}{4n}  \left( \int_{\R^n} |V|^{\crit }dx \right) \hbox{Weyl}_g\otimes B \cdot \mu_\ve^2 + o( \mu_\ve^2)
  \eal
\een
as $\ve \to 0$, where we used the Definition \ref{def:tensor:B} of $\hbox{Weyl}_g\otimes B$. Using \eqref{Sec2:19} and plugging \eqref{pohofinal2}, \eqref{pohofinal21}, \eqref{poho:78}, \eqref{pohofinal3}, \eqref{pohofinal4}  in \eqref{pohofinal1} proves \eqref{condnecessaire:5} and  concludes the proof of Theorem \ref{maintheo}.

\subsection{The case $n=4$.} It follows from \eqref{eq:expansion},  \eqref{Ve3}, \eqref{propLambda} and \eqref{estimees2} that
\begin{equation}\label{asymp:ue}
\begin{aligned}
& \lim_{R \to + \infty} \limsup_{\ve\to 0}\sup_{x\in  B(0,  R^{-1}) \backslash B(0, R \mu_\ve)}\left|\frac{|x|^{2}v_\ve(x)}{\me}-\lambda(V)\right| \text{ and }  \\
& \lim_{R \to + \infty} \limsup_{\ve\to 0}\sup_{x\in  B(0,  R^{-1}) \backslash B(0, R \mu_\ve)} \left|\frac{|x|^3 \nabla v_\ve(x)}{\mu_\ve} + 2 \lambda(V)\frac{x}{|x|} \right|=0.
\end{aligned}
\end{equation}
For any fixed $R>0$ we have, by \eqref{estimees2} and \eqref{asymp:ue}, 
$$ \int_{B(0,  \delta)}  \hat{h}_\ve v_\ve^2 dx = O(R^4 \mu_\ve^2) + \big( \hat{h}_\ve(0) + \ve_R \big)\mu_\ve^2 \int_{ B(0,  R^{-1}) \backslash B(0, R \mu_\ve)} \frac{\lambda(V)^2}{|x|^4}dx $$
where $\ve_R \to 0$ as $R \to + \infty$. As a consequence, 
$$ \int_{B(0,\delta)} \hat{h}_\ve v_\ve^2 dx= \Big(\lambda(V)^2 \omega_3 \hat{h}_\ve(0)+o(1) \Big) \me^{2}\ln\frac{1}{\me} $$
as $\ve \to 0$. By \eqref{def:hat:h} and going back to \eqref{poho:78} we therefore have, when $n=4$,
\ben \label{pohofinal3:4}
\bal 
&\int_{B(0,\delta)} \left(  x^k \partial_k v_\ve +  v_\ve \right)   \hat{h}_\ve v_\ve dx \\
 & \quad = -\lambda(V)^2\omega_{3}\left(h(x_0)-\frac{1}{6}S_g(x_0)\right)\me^{2}\ln\frac{1}{\me} +o \big( \me^{2}\ln\frac{1}{\me} \big)
\eal \een
as $\ve \to 0$, where as before we have let $x_0 = \lim_{\ve \to 0} x_\ve$. We are left with the term \eqref{term:weyl}. Integrating by parts and using \eqref{estimees2} shows that
\begin{equation} \label{weyl:dim4:1} 
\begin{aligned}
& \int_{B(0,\delta)} x^p x^q \partial^2_{ij} v_\ve \left(  x^k \partial_k v_\ve + \frac{n-2}{2} v_\ve\right) \, dx \\
& =   \int_{B(0, \delta)} x^p x^q \partial_i v_\ve \partial_j v_\ve \, dx + O(\mu_\ve^2). 
\end{aligned} 
\end{equation}
Arguing as above we get by \eqref{estimees2} and \eqref{asymp:ue} that 
$$ \begin{aligned}
   \hbox{Weyl}_{g}(x_0)_{ipqj}& \int_{B(0, \delta)}  x^p x^q \partial_i v_\ve \partial_j v_\ve \, dx \\
 & = O(R^4 \mu_\ve^2) + \hbox{Weyl}_{g}(x_0)_{ipqj}  \int_{B(0,R^{-1})\backslash B(0, R \me)} x^p x^q \partial_i v_\ve \partial_j v_\ve dx \, \\
 & = 4 \lambda(V)^2 \mu_\ve^2  \hbox{Weyl}_{g}(x_0)_{ipqj} \int_{B(0,R^{-1})\backslash B(0, R \me)} x^p x^q \frac{x_i}{|x|^4} \frac{x_j}{|x|^4} \, dx \\
 & \quad + O(R^4 \mu_\ve^2) + O\Big(\ve_R \mu_\ve^2 \ln \frac{1}{\mu_\ve}\Big) \\
 & = O(R^4 \mu_\ve^2) + O\Big(\ve_R \mu_\ve^2 \ln \frac{1}{\mu_\ve}\Big)\\
\end{aligned} $$
where as before $\lim_{R \to + \infty} \ve_R = 0$ and where the last term involving $\hbox{Weyl}_g(x_0)$ vanishes by antisymmetry of $\hbox{Weyl}_g(x_0)$. Going back to \eqref{term:weyl} with \eqref{weyl:dim4:1} we have proven that 
\begin{equation}  \label{weyl:dim4:2} 
 \int_{B(0,\delta)} \left(  x^k \partial_k v_\ve + \frac{n-2}{2} v_\ve\right)   A_\ve(v_\ve) dx \\ = o \Big( \mu_\ve^2 \ln \frac{1}{\mu_\ve} \Big) 
 \end{equation}
as $\ve \to 0$. Remark that by \eqref{deflambda} we have $2 \omega_3 \lambda(V) = \int_{\R^4} |V|^2 V dx$. With this observation, using \eqref{Sec2:19} and plugging  \eqref{pohofinal2}, \eqref{pohofinal21}, \eqref{poho:78}, \eqref{pohofinal3:4} and  \eqref{weyl:dim4:2}  in \eqref{pohofinal1}, we get \eqref{condnecessaire:4}, which proves Theorem \ref{maintheo2} for $n=4$.

\subsection{The case $n=3$.}\label{sec:3D}
We are left with the case $n=3$. Using \eqref{pohofinal2} and \eqref{poho:35} the Pohozaev identity \eqref{pohofinal1} writes as

\begin{equation} \label{poho:n3}
\begin{aligned}
&\int_{\partial B(0,\delta)} \left(  \frac{1}{2} v_\ve \partial_{\nu} v_\ve - \frac{\delta}{2} |\nabla v_\ve|_\xi^2 + \delta (\partial_\nu v_\ve)^2 + \frac{1}{p_\ve} \Lambda_{x_\ve}^{p_\ve-6} \delta |v_\ve|^{p_\ve} \right)   d\sigma \\
&= \frac{1}{12}   (6 - p_\ve)\int_{\R^n} |V|^{6}dx + o\left(  |p_\ve-6 |\right)  \\
&+ \int_{B(0,\delta)} \big(  x^k \partial_k v_\ve + \frac{1}{2} v_\ve\big)   \tilde{A}_\ve(x)\, dx
\end{aligned}
\end{equation}
with 
$$\tilde{A}_\ve(x):= \left(  \Delta_{g_\ve} - \Delta_\xi\right)   v_\ve + \Big(\frac{1}{8} S_{g_\ve}+ \hat{h}_\ve \Big) v_\ve, 
$$
where $\hat{h}_\ve$ is given by \eqref{def:hat:h} and where $v_\ve$ is defined in \eqref{def:ve}. First, by using \eqref{estimees2} together with \eqref{est:diff}, we see that 
$$ |\tilde{A}_\ve(x)| \lesssim \frac{\mu_\ve^{\frac12}}{\mu_\ve + |x|} \quad \text{ for } x \in B(0, \delta). $$
As a consequence, straightforward computations show that 
\begin{equation} \label{poho:n31}
\left|  \int_{B(0,\delta)} \big(  x^k \partial_k v_\ve + \frac{1}{2} v_\ve\big)   \tilde{A}_\ve(x)\, dx
 \right| = O \big( \delta \mu_\ve \big)
\end{equation}
as $\ve \to 0$. Independently, \eqref{estimees2} shows that 
\begin{equation} \label{poho:n32}
\left| \int_{\partial B(0,\delta)} \frac{1}{p_\ve} \Lambda_{x_\ve}^{p_\ve-6} \delta |v_\ve|^{p_\ve}   d\sigma \right| = o(  \mu_\ve)
\end{equation}
as $\ve \to 0$. We now compute the remaining boundary term. It follows from \eqref{Ve3}, \eqref{propLambda}, \eqref{estimees2}, \eqref{eq:neweq}  and standard elliptic theory that
\begin{equation}\label{lim:hua:G}
\lim_{\ve\to 0}\frac{v_\ve(x)}{\me ^{\frac{1}{2}}} =\lambda  \hat{G}(x) \quad \text{ in } C^{2 }_{loc}(B_{2\delta}(0) \backslash \{0\}),
\end{equation} 
where we have let 
\begin{equation} \label{def:lambdan3}
\lambda :=4 \pi \lambda(V)
\end{equation}
and
\begin{equation*}
\hat{G}(x):= \frac{G_h(x_0,\hbox{exp}_{x_0}^{g_{x_0}}(x))}{\Lambda_{x_0}(\hbox{exp}_{x_0}^{g_{x_0}}(x))} \quad \hbox{ for all }x\in B_{2\delta}(0) \backslash \{0\}.
\end{equation*}
Here we recall that $\lambda(V)$ is the first term in the expansion at infinity of $V$ (see \eqref{eq:expansion}), $G_h$ is the Green's function for the operator $\Delta_g+h$ in $M$, $\Lambda_{x_0}$ is as in \eqref{propLambda} and $\hbox{exp}_{x_0}^{g_{x_0}}$ is the exponential map at $x_0$ with respect to $g_{x_0}$. 
It first follows from \eqref{lim:hua:G} that
\begin{equation} \label{poho:n33}
\begin{aligned} \lim_{\ve\to 0}& \me^{-1}\int_{\partial B(0,\delta)} \left(  \frac{1}{2} v_\ve \partial_{\nu} v_\ve - \frac{\delta}{2} |\nabla v_\ve|_\xi^2 + \delta (\partial_\nu v_\ve)^2  \right)   d\sigma\\
=&\lambda^2\int_{\partial B(0,\delta)} \left(  \frac{1}{2} \hat{G} \partial_{\nu} \hat{G} - \frac{\delta}{2} |\nabla \hat{G}|_\xi^2 + \delta (\partial_\nu \hat{G})^2 \right)   d\sigma.
\end{aligned}
\end{equation}
By \eqref{propLambda} and \eqref{DLmetriqueconf}, and since $G_h(x_0, \cdot)$ satisfies the expansion \eqref{DL:foncgreen}, there exists a continuous function $\hat{\beta}: B(0,2 \delta) \to \R$ such that  
\begin{equation*}
\hat{G}(x)= \frac{1}{4\pi |x|}+\hat{\beta}(x) \quad \hbox{ for all }x\in B_{2\delta}(0) \backslash \{0\}. 
\end{equation*}
It follows from Definition \ref{def:mass} that $\hat{\beta}(0) = m_{h}(x_0)$ and \eqref{est:C1:masse} and \eqref{DLmetriqueconf} show that we have 
\begin{equation*}
|\hat{\beta}(x)|\leq C\, ;\, |\nabla \hat{\beta}(x)|\leq C(1+|\ln |x||) 
\quad \hbox{ for all }x\in B_{2\delta}(0) \backslash \{0\}.
\end{equation*}
Straightforward computations then show that 
\begin{equation} \label{poho:n34}
\int_{\partial B(0,\delta)} \left(  \frac{1}{2} \hat{G} \partial_{\nu} \hat{G} - \frac{\delta}{2} |\nabla \hat{G}|_\xi^2 + \delta (\partial_\nu \hat{G})^2  \right)   d\sigma 
 = -\frac{1}{2}m_h(x_0)+\ve_\delta, 
\end{equation}
where $\lim_{\delta \to 0} \ve_\delta = 0$. Using \eqref{Sec2:19}, plugging \eqref{poho:n31}, \eqref{poho:n32}, \eqref{poho:n33} and \eqref{poho:n34} in \eqref{poho:n3} and using \eqref{def:lambdan3} finally shows that 
\begin{equation} \label{poho:n35}
\lim_{\ve\to 0}\frac{6 - p_\ve}{\me} =\frac{-96\pi^2\lambda(V)^2 m_{h}(x_0)}{   \int_{\R^3} |V|^{6}dx}. 
\end{equation}
This proves \eqref{condnecessaire:3} and concludes the proof of Theorem \ref{maintheo2}. $\square$

\section{A non-blowup situation}\label{sec:final}
In this last Section we show how Theorem \ref{maintheo} can be used to rule out the existence of one-bubble blowing-up solutions of equations like:
\begin{equation} \label{eq:dmw}
 \Delta_g u_\ve + h u_\ve = |u_\ve|^{\crit-2-\ve} u_\ve
 \end{equation}
at a given point, in some cases. In \cite{DengMussoWei}, Deng-Musso-Wei have proposed a creative approach to construct families of sign-changing solutions $(u_\ve)_{\ve >0}$ to the equation \eqref{eq:dmw} that blow-up like a single bubble as in \eqref{one:bubble}. The bubble that the authors choose in \cite{DengMussoWei} is one of the symmetric bubbles constructed in \cite{DelPinoMussoPacardPistoia1}. In their arguments, Deng-Musso-Wei have introduced the function $$\vp_\ell: = h - c_n \Big( 1 + \frac{n-4}{3n}\ell  \Big) S_g,$$
where $\ell$ is a parameter associated to the familie of bubbles in \cite{DelPinoMussoPacardPistoia1}. The function $\vp_\ell$ plays a crucial role in the results of \cite{DengMussoWei}. We show below how, under some conditions, it is possible to use $\vp_\ell$ to rule out one-bubble blow-up.

\medskip\noindent We let $(M,g)$ be a smooth, closed, connected Riemannian manifold of dimension $n \ge 5$ that is locally conformally flat and has constant scalar curvature equal to $-1$. This is for instance satisfied if $(M,g)$ is a closed hyperbolic manifold. We define  
$$c_n = \frac{n-2}{4(n-1)}\hbox{ and  we let }1<t<1+\frac{n-4}{3n}.$$
Let $\xi_0 \in M$ be any point.

\smallskip\noindent{\bf Step \ref{sec:final}.1:} We claim that there exists $h \in C^\infty(M)$ such that 
\begin{equation} \label{hyph}
\left \{ \begin{aligned}
& h(\xi_0) = - tc_n \\
& \nabla h(\xi_0) = 0 \\
& \nabla^2 h(x_0) \text{ is positive-definite } \\
& \Delta_g + h \text{ is coercive in } H^1(M).
\end{aligned} \right.
\end{equation}
Under \eqref{hyph}, $\xi_0$ is in particular a local strict minimum point of $h$. 

\smallskip\noindent We prove the Claim. Let for instance  $\chi \in C^\infty_c(\R^n)$ be a cutoff function with $0 \le \chi \le 1$,  $\chi(x)=1$ in $B(0, 1)$ and $\chi(x)=0$ in $\R^n \backslash B(0, 2)$. Let, for $0<\delta < \frac{i_g(M)}{2}$ and $x \in M$, 
$$ h_{\delta}(x) =  1 - \chi\Big( \frac{d_g(\xi_0, x)}{\delta} \Big) +  \chi\Big( \frac{d_g(\xi_0, x)}{\delta} \Big)  \Big( - tc_n + d_g(\xi_0, x)^2 \Big). $$
We claim that $h_\delta$ satisfies \eqref{hyph} for $\delta$ small enough (but fixed). Indeed, in $B_g(\xi_0, \delta)$, we have $h_\delta(x) =  - tc_n + d_g(\xi_0, x)^2$, and straightforward computations show that 
$$ \Vert h_{\delta} - 1 \Vert_{L^{\frac{n}{2}}(M)} \le C \delta^2, $$   
so that $\Delta_g + h_\delta$ is coercive for $\delta$ small enough. This proves the claim.

\smallskip\noindent{\bf Step \ref{sec:final}.2:} We now let $h \in C^\infty(M)$ satisfy \eqref{hyph}. For $\ell \ge 1$ and $\xi \in M$ we define
$$\begin{aligned}
 \vp_\ell(\xi) & = h(\xi) - c_n \Big( 1 + \frac{n-4}{3n}\ell  \Big) S_g(\xi)  = h(\xi) + c_n \Big( 1 + \frac{n-4}{3n}\ell \Big), 
  \end{aligned} $$
We claim that
 \begin{equation}\label{ppty:phi:l}
 \left\{\begin{array}{l}
 \Delta_g+h\hbox{ is coercive};\\
\vp_\ell(\xi_0) >0\hbox{ for all }l\geq 1;\\
\xi_0\hbox{ is a $C^1-$stable critical point of }\vp_\ell\hbox{ \emph{uniformly in }}\ell \ge 1.
\end{array}\right\}\end{equation}
We prove the claim. We have that
$$ \vp_\ell(\xi_0) =c_n \Big( 1 + \frac{n-4}{3n}\ell -t \Big)\geq c_n \Big( 1 + \frac{n-4}{3n} -t \Big) >0 $$
for any $\ell \geq 1$. It is easily seen with \eqref{hyph} that, for any $\ell \ge 1$, $\xi_0$ is a critical point of $\vp_\ell$ satisfying $\vp_\ell(\xi_0) >0$
 and $\vp_\ell(\xi_0) \to + \infty$ as $\ell \to + \infty$. Since $\xi_0$ has non-zero degree, it is a $C^1$-stable critical point of $h$ (following the terminology of \cite{LiStable}). That is, there is $r_0 >0$ that satisfies the following: for any $0< r \le r_0$ there exists $\alpha >0$ such that any $\tilde h \in C^1(\overline{B_g(\xi_0, r)})$ satisfying  $\Vert \tilde h - h \Vert_{C^1(\overline{B_g(\xi_0, r)})} \le \alpha$ has at least one critical point in $B_g(\xi_0, r)$. By definition of $\vp_\ell$ we observe that $\xi_0$ then remains a $C^1$-stable critical point of $\vp_\ell$   \emph{uniformly in $\ell \ge 1$:} that is, for any $0 < r \le r_0$ and $\alpha$ as above, and for any  $\ell \ge 1$, any function $\tilde \vp \in C^1(\overline{B_g(\xi_0, r)})$ satisfying  $\Vert \tilde \vp - \vp_\ell \Vert_{C^1(\overline{B_g(\xi_0, r)})} \le \delta$ has at least one critical point in $B_g(\xi_0, r)$. This proves \eqref{ppty:phi:l}.

\medskip\noindent{\bf Step \ref{sec:final}.3:} We claim that there is no family  $(u_\ve)_{0< \ve \le 1}\in C^2(M)$ of solutions of \eqref{eq:dmw} such that 
\begin{equation}\label{one:bubble:bis}
u_\ve =B_\ve+o(1)\hbox{ in }H^1(M),
\end{equation}
where $B=(B_\ve)_{\ve>0}$ is a bubble centered at $(\xi_{\ve})_\ve\in M$, $\lim_{\ve\to 0}\xi_\ve=\xi_0$ as in Definition \ref{def:bubble}.

\smallskip\noindent We prove the claim by contradiction. Since $(M,g)$ is locally conformally flat, then the Weyl tensor vanishes. It then follows from Theorem  \ref{maintheo} of the present paper that if there is blow-up like in \eqref{one:bubble:bis}, then equation \eqref{condnecessaire:5} in particular tells us that 
\begin{equation} \label{contra}
  h(\xi_0) + c_n = h(\xi_0) - c_n S_g(\xi_0) \ge 0. 
  \end{equation}
But this is an obvious contradiction with \eqref{hyph} since $h(\xi_0) = - t c_n$ and $t >1$. This proves the claim.
 
\bibliographystyle{alpha}
\bibliography{biblio}

\end{document}